\documentclass[12pt]{article}
\usepackage{amsmath}
\usepackage{graphicx}
\usepackage{enumerate}
\usepackage{natbib}
\usepackage{url} % not crucial - just used below for the URL

\usepackage{amsthm}
\usepackage{amsfonts}
\usepackage{mathtools}
\usepackage{amssymb}
\usepackage{color}
\usepackage{url}
\usepackage{lmodern}
\usepackage{anyfontsize}

\newtheorem{theorem}{Theorem}

\newtheorem{remark}{Remark}
\newtheorem{assumption}{Assumption}
\newtheorem{lemma}{Lemma}

\newcommand{\blind}{1}

\begin{document}

\def\spacingset#1{\renewcommand{\baselinestretch}%
{#1}\small\normalsize} \spacingset{1}

%%%%%%%%%%%%%%%%%%%%%%%%%%%%%%%%%%%%%%%%%%%%%%%%%%%%%%%%%%%%%%%%%%%%%%%%%%%%%%

\if1\blind
{
  \title{\bf Latent Model Extreme Value Index Estimation}
  \author{Joni Virta\thanks{
    Joni Virta gratefully acknowledges the financial support from the Academy of Finland (Grant 321883). Niko Lietz\'{e}n gratefully acknowledges the financial support from the Emil Aaltonen Foundation (Grant
    190135~N). All authors acknowledge the computational resources provided by the Aalto Science-IT project.}\hspace{.2cm}\\
	%Department of Mathematics and Statistics,
	University of Turku\\
    %Department of Mathematics and Systems Analysis,\\
    Aalto University School of Science\\
    and \\
    Niko Lietz\'en \\
    %Department of Mathematics and Systems Analysis,\\
    Aalto University School of Science\\
	and \\
Lauri Viitasaari \\
Aalto University School of Business \\
	and \\
	Pauliina Ilmonen \\
	Aalto University School of Science}

  \maketitle
} \fi

\if0\blind
{
  \bigskip
  \bigskip
  \bigskip
  \begin{center}
    {\LARGE\bf Latent Model Extreme Value Index Estimation}
\end{center}
  \medskip
} \fi

\bigskip
\begin{abstract}
We propose a novel strategy for multivariate extreme value index estimation. In applications such as finance, volatility and risk present in the components of a multivariate time series are often driven by the same underlying factors, such as the subprime crisis in the US. To estimate the latent risk, we apply a two-stage procedure. First, a set of independent latent series is estimated using a method of latent variable analysis. Then, univariate risk measures are estimated individually for the latent series to assess their contribution to the overall risk. As our main theoretical contribution, we derive conditions under which the effect of the first step to the asymptotic behavior of the risk estimators is negligible. Simulations demonstrate the theory under both i.i.d. and dependent data, and an application into financial data illustrates the usefulness of the method in extracting joint sources of risk in practice.
\end{abstract}

\noindent%
\vfill

\newpage
%\spacingset{1.5} % DON'T change the spacing!

\section{Introduction}\label{sec:introduction}

Let $ \textbf{x}_1, \ldots , \textbf{x}_n $ be a sample of $ p $-variate random vectors with possibly dependent distributions. For each observation, we assume the instantaneous latent variable model,
\begin{align}\label{eq:latent_model}
\textbf{x}_i = f(\textbf{z}_i), \quad i = 1, \ldots , n,
\end{align}
where the latent $ p $-variate random vectors  $ \textbf{z}_1, \ldots , \textbf{z}_n $ are assumed to have independent components in the sense that the $ k $th component $ z_{ik} $ of $ \textbf{z}_i $ is independent of the $ l $th component $ z_{jl} $ of $ \textbf{z}_l $, for all $ i, j $ and $ k \neq l $. Furthermore, we assume $ f: \mathbb{R}^p \rightarrow \mathbb{R}^p $ is a deterministic function that is smooth and bijective. Note that while no explicit noise term is present in \eqref{eq:latent_model}, the general formulation still captures noisy latent models as well, as one or several of the $ p $ latent components can represent noise which is then combined with the other components (signals) by the function $ f $ in a desired manner (additively, multiplicatively etc.).

The model \eqref{eq:latent_model} can be considered as a very general form of \textit{independent component analysis} and has applications in numerous fields such as in telecommunications, psychometrics, economics and finance \citep{comon2010handbook, hyvarinen2000independent}. The model provides a powerful alternative to standard multivariate modelling schemes as, after having estimated the latent vectors, the independence of their components implies that all subsequent modeling can be done univariately. This structural simplification leads to both smaller number of parameters to estimate and simplified interpretations for the components as no interactions between the series need to be acknowledged.

%Especially in economics and finance, a central theme is assessing the magnitudes and likelihoods of risks. This

\textcolor{black}{In this paper we focus on estimating the tail behaviour of the latent variables in the model~\eqref{eq:latent_model}, evaluated through the \textit{extreme value indices} of the corresponding distributions \citep{de2007extreme}. This is a natural goal to pursue in many financial and signal processing applications as the heaviness of the tails of a distribution is an indicator of an unstable and risky signal. For example, the independent component model has been applied in cashflow analysis and prediction of financial time series data \citep{kiviluoto1998independent, lu2009financial, yang2005fast}, and in this context assessing the tail behaviour of the obtained independent components could help identify common sources of financial risk. Similarly, the evaluation of the extreme behaviour of latent components could help identify the sources of abnormalities in applications such as biomedical imaging \citep{roberts2000extreme} or maritime vessel track analysis \citep{smith2012online}.} %Assessing the extreme behaviour of a random sample boils down to estimating the so called extreme value index of the corresponding distribution \citep{de2007extreme}.} %, However, in the aforementioned articles, the authors did not consider risk assessment. In this article, we examine extreme behaviour of independent components. This approach gives us the advantage of being able   Measuring risk is closely related to the estimation of extreme quantiles, as estimating the probability of a catastrophic event often involves extrapolating outside of the empirical distribution. For example, an investor might be interested in estimating the probability that a financial loss exceeds some chosen level outside the range of the observed data. \joni{signal processing example as well.

%\joni{In Figure REF we have plotted the cashflow data of XXX different companies over the course of XXX days. (Add the plot here along with more explanantion of the data). Large fluctuations in the time series correspond to higher values of risk and it is reasonable to assume that the actual risk factor is not specific to any single company but it is something underlying the entire economy. The true measure of risk can then be modelled as one of the latent variables in the model \eqref{eq:latent_model}. Independent component model has been applied in cashflow analysis and prediction of financial time series data \cite{kiviluoto1998independent, lu2009financial, yang2005fast}. However, in the aforementioned articles, the authors did not consider risk assessment. In this article, we examine extreme behaviour of independent components. This approach gives us the advantage of being able to identify multiple common sources of risk, while still resorting only to univariate techniques of extreme value analysis. For example, in the cashflow data, displayed in Figure REF, it could be that the risk is partially caused both by XXX and YYY, these two factors acting independently of each other. Our proposed method is used to analyze this dataset in Section \ref{sec:examples} where (some explanation here later).}

\textcolor{black}{This objective can be reached in two steps. First, we estimate a mapping $ \hat{f}^{-1} $ such that $  \hat{f}^{-1} (\textbf{x}) $ equals the latent components up to order and scales (in latent component analysis, the order and scales of the latent components are usually neither of interest nor identifiable, see, e.g., \cite{tong1991indeterminacy}). Especially under linear $ f $, numerous techniques for obtaining consistent estimators under various types of data exist, \textcolor{black}{see Section~\ref{sec:two_models} for examples}. Second, after having obtained the sample estimates $ \hat{f}^{-1}(\textbf{x}_1), \ldots , \hat{f}^{-1}(\textbf{x}_n)$ of the latent vectors, we use one of the several univariate extreme value index estimators presented in the literature \citep{de2007extreme} to assess the extreme behavior of the individual, now independent, components.} %The process can then be viewed as a concatenation of two consecutive approximations. %the empirical distribution $ \hat{F}_n $ of the estimated latent variables is an approximation to the empirical distribution $ F_n $ of the true latent variables. $ F_n $ in turn is an approximation of the population distribution $ F $ of the latent variables.

\textcolor{black}{Note that, in contrast to the above, the standard approach in multivariate extreme value theory is to assess the extreme behaviour component-wise for the observed multivariate signal itself \citep{de2007extreme}. Approaches which in some way acknowledge the multivariate structure of the data have been proposed only recently, and they include considering convex combinations of the component-wise estimators \citep{dematteo2016tail, kim2017estimation}, extreme risk region estimation \citep{cai2011estimation}, and estimating the extreme value index of the generating variate of an underlying elliptical model \citep{dominicy2017multivariate, heikkila2019multivariate}. However, these methods either involve complicated estimation or require strict distributional assumptions, making them less than ideal in practice. In comparison, our proposed two-step procedure is straightforward to apply and takes the multivariate form of the data into account in a natural way. Moreover, the associated latent variable model is flexible, allowing different tail behaviors for the underlying independent components. The only structural assumption we make is that the observed variables are generated by a set of independent factors.}

\subsection{Scope and structure of the paper}

\textcolor{black}{Of the two steps of our proposed method, we are primarily interested in the latter.} That is, we focus on assessing the extreme behavior of the individual components in the independent component model \eqref{eq:latent_model}.  Throughout the article, we assume that there exists an estimator $ \hat{f}^{-1} $ with the asymptotic linearization
\begin{align}\label{eq:linearization}
\hat{\textbf{z}}_i: = \hat{f}^{-1} (\textbf{x}_i) = \textbf{z}_i + \hat{\textbf{H}} \textbf{z}_i + \hat{\textbf{r}},
\end{align}
where the $ p \times p $ -matrix $ \hat{\textbf{H}} = \mathcal{O}_p(c_n^{-1}) $ and the $ p $-vector $ \hat{\textbf{r}} = \mathcal{O}_p(c_n^{-1})$ for some rate $ c_n $. Here $\mathcal{O}_p(c_n^{-1})$ denotes the element-wise ``convergence rate'' in probability. For a precise definition, see Section \ref{sec:general_framework}. The form \eqref{eq:linearization} is very general and encompasses many popular estimators $ \hat{f}^{-1} $ in the independent component analysis and blind source separation literature, \textcolor{black}{see Section~\ref{sec:two_models} for examples.} %Two common special cases of \eqref{eq:linearization} along with the corresponding exact forms of $ \hat{f}^{-1} $ are discussed in Section \ref{sec:two_models}.
\textcolor{black}{Assuming for now that an estimator $ \hat{f}^{-1} $ exists in the sense of \eqref{eq:linearization}, our main objective is to estimate the extreme value indices of the components of the latent variables using $ \hat{\textbf{z}}_i $ as a proxy for $ \textbf{z}_i $, and to show that this approximation incurs no loss in asymptotic efficiency under a suitable set of assumptions.}

\textcolor{black}{A further complicating factor is that latent variable models such as \eqref{eq:latent_model} are well-known for not having fixed signs or scales for the latent components. That is, the vector $ \textbf{z}_i $ on the right-hand side of \eqref{eq:linearization} actually corresponds in many models to the true latent vectors only up to the signs and scales of its components.} In the standard usage of latent variable modelling this is most often acceptable, as our interest lies commonly not in the signs, but in the shapes of the distributions of the latent variables. \textcolor{black}{Similarly, in the present context, the scale of the components is irrelevant as most commonly applied extreme value index estimators are scale-invariant.} However, as risk is estimated from the tails of the components, knowing in which of the tails we are in is for our purposes of paramount importance, and we need a way of identifying the correct tail. \textcolor{black}{A simple, but restrictive, solution would be to require that all the latent components have symmetric distributions.} Instead, we choose to assess the extreme behaviour of, not the latent components, but their absolute values. This rids us of the sign indeterminacy by ``stacking'' the two tails on top of each other. Since the absolute value inherits its tail behaviour from the heavier of the two tails, this approach has the interpretation of us always looking at the heavier of the two tails. Moreover, as heavier tails correspond to larger risk, the use of absolute values can be seen as a conservative approach to tail behaviour estimation.

The rest of the paper is organized as follows:  Preliminaries on extreme value theory along with the popular extreme value index estimators, the Hill estimator and the moment estimator, are reviewed in Section \ref{sec:extreme_value}. These extreme value index estimators are known to be consistent and asymptotically normal under mild technical conditions. In Section \ref{sec:general_framework}, we derive sufficient conditions ensuring that the asymptotic properties of the extreme value estimators are preserved when estimated using the proxy sample $ \hat{f}^{-1} (\textbf{x}_i) $. %\joni{We show that the condition essentially requires that the convergence rate of the estimator $\hat{f}^{-1}$ must be sufficiently fast with respect to a certain tuning parameter of the Hill estimator and that the latent components must not be too heavy. The strictness of the latter condition can be relaxed with the tuning parameter with the trade-off that the extreme value index estimator will converge at a slower rate.}
In Section \ref{sec:two_models}, we consider two example cases of the general framework and discuss the particular assumptions needed to achieve the limiting results for the corresponding proxy samples. In Section \ref{sec:examples}, we present a large simulation study and a real data application is considered in Section \ref{sec:real_data}. %We end the paper with discussions,
All the proofs are postponed to the supplementary appendix, along with a supplementary simulation study and additional details concerning the real data example.

\section{Preliminaries on extreme value theory}\label{sec:extreme_value}

% Extreme value theory is the study of extreme phenomena, with application areas varying from telecommunications, biostatistics, climatology, and geology  to finance and beyond. For a comprehensive review of classical approaches to extreme value theory, see \cite{de2007extreme}.
% In this article, our focus is on extreme value index estimation for which we next review some necessary preliminaries.  %One of the main applications of the theory is extreme quantile estimation, or its dual problem, tail probability estimation.
%The classical results include classification of
%distributions according to their tail behavior and estimators that are used to
%infer these properties from data.
In the following we provide a brief introduction to the topics in univariate extreme value theory that are most relevant to our objectives. See \cite{de2007extreme} and the references therein for more information.

Consider an i.i.d. random sample $\textbf{y} = (y_1, \ldots , y_n)$ from a univariate distribution $F$ and the sample maximum $M_n = \max_{1 \leq i \leq n} y_i$. If there exists sequences of constants $a_n > 0$ and $b_n$ such that $a_n M_n+ b_n$ has a limiting distribution $G$, we say that $G$ is the \emph{extreme value distribution} of $F$. One of the fundamental results in extreme value theory is the Fisher-Tippett-Gnedenko theorem \textcolor{black}{which identifies the class of distributions~$G$.}

\begin{theorem}[Fisher-Tippett-Gnedenko]\label{thm:ftg}
	The class of extreme value distributions is $G_\gamma \left( a x + b \right)$ with $a > 0$ and $b\in \mathbb{R}$, where
	$$
	G_\gamma\left(x \right)= \exp\left(-\left(1+ \gamma x \right)^{-1/\gamma} \right), \quad 1+ \gamma x > 0,
	$$
	with $\gamma \in \mathbb{R}$ and where for $\gamma = 0$ the right-hand side is interpreted as $\exp\left(-e^{-x} \right)$.
\end{theorem}

According to Theorem \ref{thm:ftg}, the family of possible extreme value distributions has a remarkably simple form, parametrized by a single real number $\gamma$. If $G_\gamma$ is the extreme value distribution of $F$, the distribution $F$ is said to be in the \emph{domain of attraction} of $G_\gamma$, and we write $ F \in G_\gamma $. The parameter $\gamma$ is said to be the \emph{extreme value index} of $F$. The \textcolor{black}{parameter} $ \gamma $ measures the thickness of the (right) tail of $ F $ and knowing its value leads to a complete characterization of the asymptotic tail behavior of $ F $, allowing extrapolating probabilities beyond the observed dataset. Thus $ \gamma $ is a key ingredient in risk assessment.

It is widely accepted that distributions are divided into heavy and light tailed ones based on the sign of $\gamma$. More precisely, for $\gamma>0$, the distributions $F\in G_\gamma$ are called \emph{heavy tailed} and belonging to the domain of attraction of the \emph{Frechet} distribution. Similarly, if $\gamma<0$ and $F\in G_\gamma$, then we say that $F$ is \emph{light tailed} and belongs to the domain of attraction of the \emph{Weibull} distribution. Finally, if $F \in G_0$, then $F$ belongs to the domain of attraction of the \emph{Gumbel} distribution. This corresponds to the border case between light and heavy tails, and includes, e.g., the case of a normal distribution.

One of the most commonly applied classical estimators of the extreme value index, suitable for $\gamma > 0$, is the Hill estimator introduced in \cite{hill},
$$
\hat{\gamma}_H(\textbf{y}) = \frac{1}{k_n}\sum\limits_{m=0}^{k_n-1} \log \frac{(\textbf{y})_{(n-m,n)}}{(\textbf{y})_{(n-k_n,n)}},
$$
where $(\textbf{y})_{(n,n)} \geq \dots \geq (\textbf{y})_{(1, n)}$ are the order statistics of the sample $ \textbf{y} $, and $1 \leq k_n \leq n$ is a sequence of thresholds for the portion of observations that are considered to form the tail. Common choices for the threshold include, e.g., $ k_n = \sqrt{n} $ and $ k_n = \log(n) $.

Another well-known estimator, which in turn is valid for any value of $\gamma$, is the moment estimator introduced in \cite{momentConsistency}. As in the Hill estimator, set
$$
M_n^{(j)}(\textbf{y}) = \frac{1}{k_n}\sum\limits_{m=0}^{k_n-1} \left( \log \frac{(\textbf{y})_{(n-m,n)}}{(\textbf{y})_{(n-k_n,n)}} \right)^j.
$$
Here $j=1,2,\ldots$ is given, and the Hill estimator corresponds to the choice $j=1$. The moment estimator is based on the choices $j=1,2$ and is given by
$$
\hat{\gamma}_M(\textbf{y}) = M_n^{(1)}(\textbf{y}) + 1 - \frac{1}{2}\left( 1-\frac{[M_n^{(1)}(\textbf{y})]^2}{M_n^{(2)}(\textbf{y})} \right)^{-1}.
$$
\textcolor{black}{In the next section, both the Hill estimator and the moment estimator are used to estimate the extreme value indices of the absolute values of the latent components in \eqref{eq:latent_model}.}

%The estimators have simple consistency properties for such $k_n$:
%
%\begin{theorem}[\cite{hillConsistency}]\label{thm:consistency_hill}
%	Let $y_1, y_2, \dots$ be i.i.d. random variables with common distribution $F \in \mathcal{D}(G_\gamma)$ with
%	$\gamma >0$. Let $k_n/n \to 0, k_n \to \infty$ as $n \to \infty$. Then
%	$$
%	\hat{\gamma}_H \to_P \gamma.
%	$$
%\end{theorem}
%A result under similar conditions holds for the moment estimator, however, the parameter
%$\gamma$ is allowed to have any value.
%
%\begin{theorem}[\cite{momentConsistency}]\label{thm:consistency_moment}
%	Let $y_1, y_2, \dots$ be i.i.d. random variables with common distribution $F \in \mathcal{D}(G_\gamma)$.
%	Assume that $\sup\{y : F(y)<1\}>0$.
%	Let $k_n/n \to 0, k_n \to \infty$ as $n \to \infty$. Then
%	$$
%	\hat{\gamma}_M \to_P \gamma.
%	$$
%\end{theorem}

\section{Extreme value index estimation for latent variables}\label{sec:general_framework}
\textcolor{black}{Recall from Section \ref{sec:introduction} that we} consider an estimated sample $ \hat{\textbf{z}}_1, \ldots , \hat{\textbf{z}}_n$ of the latent vectors $ \textbf{z}_1, \ldots , \textbf{z}_n$ satisfying
\begin{align} \label{eq:general_model}
\hat{\textbf{z}}_i = \textbf{z}_i + \hat{\textbf{H}} \textbf{z}_i + \hat{\textbf{r}},
\end{align}
where the $ p \times p $ -matrix
$ \hat{\textbf{H}} = \mathcal{O}_p(c_n^{-1}) $, and the $ p $-vector $ \hat{\textbf{r}} = \mathcal{O}_p(c_n^{-1})$ for some rate $ c_n $. Here, and throughout the paper, the notation $X_n = \mathcal{O}_p(g_n)$ \textcolor{black}{is used to denote that the family of random variables $ g_n^{-1}X_n$ is uniformly tight.} Similarly, we use other Landau notation, such as $o(1)$ to indicate convergence towards zero. With $\rightarrow_p$, we denote convergence in probability, and with $\rightsquigarrow$, we indicate weak convergence, i.e., convergence in distribution.

\textcolor{black}{Recall further, that the} idea underlying the model \eqref{eq:general_model} is that the vector $ \hat{\textbf{z}}_i $ is an estimate of $ \textbf{z}_i $ obtained by solving some latent variable model. However, for the following results to hold, simply having the form \eqref{eq:general_model} is sufficient, regardless of how it originated.

A common assumption in extreme value literature as well as in the latent variable literature is to assume that each component $z_i^k$ of the true non-observable signals $\textbf{z}_i$ is strictly stationary, and has a univariate marginal $F_k$, i.e., each observation $z_i^k$ has marginal distribution $F_k$, for all $i$. One typical example is the case where observations are i.i.d., with components drawn from different distributions. Another typical example is the case where the components $z_i^k$ form different stationary series with marginals $F_k$. \textcolor{black}{However, while our main examples arise from stationary series falling into the above setting, our main results do not even require stationarity of the components $z^k$ (although it might be difficult to interpret the estimated extreme value index if the one dimensional marginals are not equal).}

It is also customary in the field of extreme value theory to assume that marginals do not have point-mass at zero. %In our case this is to say that each component of each observation, $z_i^k$, $k=1,\ldots, p$, $i=1,\ldots,n$, satisfies $\mathbb{P}(z_i^k = 0)=0$. As we are interested in the absolute values $|z^k|$, this assumption would ensure that we can take the logarithms needed for the extreme value estimators.
%In order to ensure that our logarithm-based estimators are well-defined, we require that the marginals of the components $z_i^k$ have no mass at zero.
\textcolor{black}{In our case, this ensures that our logarithm-based estimators are well-defined.}
That is, in the general non-stationary case, we assume that, for each component $k=1,2,\ldots,p$, we have
\begin{equation}
\label{eq:general-no-zero-mass}
\lim_{\delta \to 0}\inf_{i\geq 1}\mathbb{P}\left(|z_i^k|\geq \delta\right) = 1.
\end{equation}
In the sequel, \eqref{eq:general-no-zero-mass} is always assumed, even if it is not explicitly stated. Note that \eqref{eq:general-no-zero-mass} is a natural assumption and not very restrictive. First of all, \eqref{eq:general-no-zero-mass} implies that $\mathbb{P}(z_i^k = 0)=0$ for all $i$ and $k$. Moreover, in the case of equal marginals, \eqref{eq:general-no-zero-mass} is equivalent to $\mathbb{P}(z_i^k = 0)=0$. In the general case, \eqref{eq:general-no-zero-mass} excludes also the \textcolor{black}{situations} where \textcolor{black}{the} observations come from a sequence of distributions $F_{i,k}$ that approach a distribution having point mass at zero.

\textcolor{black}{In the sequel}, the notation $ | \textbf{z}^k | $ refers to the sample $ |z_{1k}|, \ldots , |z_{nk}| $ of the absolute values of the $ k $th latent series and $ | \textbf{z}^k |_{(m, n)} $ denotes the $ m $th largest element of $ | \textbf{z}^k | $.

Throughout the article, we make the following assumption.

\begin{assumption}\label{assu:component_rates}
	For all $ k = 1, \ldots , p $, there exists deterministic sequences $ a_{nk}, b_{nk}$ for which the $ k $th component $ z_{ik} $ of $ \textbf{z}_{i} $ satisfies
	\[
	\frac{|\textbf{z}^k|_{(n, n)} - b_{nk}}{a_{nk}} = \mathcal{O}_p(1).
	\]
\end{assumption}
We stress that Assumption \ref{assu:component_rates} is very relaxed, and in the extreme value theory literature it is usually taken as granted, without explicitly stating it. Indeed, if the observations are independent with a distribution function $F$, then Assumption~\ref{assu:component_rates} follows immediately whenever $F \in G_\gamma$, i.e., $F$ is in the domain of attraction of some extreme value distribution $G_\gamma$. Thus, in the case of independent observations, discussing \textcolor{black}{the} extreme value index $\gamma$ without Assumption \ref{assu:component_rates} is not sensible. More generally, Assumption \ref{assu:component_rates} follows immediately whenever $a_{nk}^{-1}\left(|\textbf{z}^k|_{(n, n)} -b_{nk}\right)$ converges towards some distribution.  For example, Assumption \ref{assu:component_rates} is trivially valid even in the totally degenerate case $z_{ik} = z_k$, for all~$i$.

The main contribution of this article is the derivation of sufficient conditions under which the asymptotic properties of the Hill and moment estimators are  preserved under Model \eqref{eq:general_model}. Intuitively, one would expect that these asymptotic properties remain the same, provided that $c_n^{-1}$ vanishes rapidly enough to compensate the growth of the \textcolor{black}{sample} maximum of the heaviest component. Theorem \ref{prop:consistency_and_limiting} and Theorem \ref{prop:consistency_and_limiting2} below contain the precise statements of this heuristic argument.  In the sequel, we use the notation $ g_{nk} = \max \{ a_{nk}, b_{nk} \}$.
\begin{theorem}\label{prop:consistency_and_limiting}
	Let Assumption \ref{assu:component_rates} hold and assume that,
	\begin{equation}
	\label{eq:key-consistency}
	\frac{\max_\ell \{ g_{n\ell} \}}{c_n} = o(1).
	\end{equation}
	Let $k \in \{1,\ldots,p\}$ be fixed and let $C_H$ and $C_M$ be arbitrary constants.
	\begin{itemize}
		\item[i)] If $		\hat{\gamma}_H(|\textbf{z}^k|) \rightarrow_p C_H$, then
		$		\hat{\gamma}_H(|\hat{\textbf{z}}^k|) \rightarrow_p C_H$.
		\item[ii)] If $		\hat{\gamma}_H(|\textbf{z}^k|) \rightarrow_p C_H$, $\dfrac{\max_\ell \{ g_{n\ell} \}}{c_n\hat{\gamma}_H(|\textbf{z}^k|)}\rightarrow_p 0$ and $		\hat{\gamma}_M(|\textbf{z}^k|) \rightarrow_p C_M$, then
		$		\hat{\gamma}_M(|\hat{\textbf{z}}^k|) \rightarrow_p~C_M$.	
	\end{itemize}
\end{theorem}
Note that in the above result it is not required that $C_H$ and $C_M$ are the correct extreme value indices --- any constants suffice. Indeed, the part i) of Theorem \ref{prop:consistency_and_limiting} simply states that whenever the Hill estimator based on the ''true'' latent signals converges towards some constant, then the Hill estimator based on the estimated latent signals converges towards the same constant. The reason behind our formulation is that usually, as is the case for independent observations, the Hill estimator converges towards $\max(0,\gamma)$, see \cite{de2007extreme}, pp. 101. In other words, the Hill estimator vanishes for distributions that are not heavy tailed. For such distributions, one can then apply the moment estimator. Part ii) of Theorem \ref{prop:consistency_and_limiting} says that whenever both, the Hill estimator and the moment estimator based on the true latent signals $|\textbf{z}^k|$, converge towards any constants, then the Hill and the moment estimator based on the estimated latent signals converge towards the same constants. As in most cases the Hill estimator converges towards $\max(0,\gamma)$, and does not explode, part ii) of Theorem \ref{prop:consistency_and_limiting} implies that asymptotic properties of the moment estimator are inherited to the estimated model as well. The extra condition in part ii) concerns the case when the Hill estimator converges towards zero\textcolor{black}{, $ C_H = 0$, } and ensures that this convergence is not too rapid in comparison to the growth of the heaviest tail. In many cases of interest, the convergence rate of the Hill estimator is $\sqrt{k_n}$. This leads to the same condition as in Theorem~\ref{prop:consistency_and_limiting2}, and can be achieved by a suitable choice of $k_n$. Finally, we stress that  an  examination of the proof of Theorem~\ref{prop:consistency_and_limiting} reveals that the item ii) is valid as long as the Hill estimator does not tend to infinity. Thus one can safely apply the moment estimator for light tailed distributions under Model~\eqref{eq:general_model}. %\joni{I conjecture that for stationary sequences with light-tailed distributions, Hill-estimator always converge towards zero. Reasoning is that this is the case on both extremal cases: with independent observations, and with totally degenerate case (Compare to Lemma 4 that holds for arbitrary stationary sequence)}

In order to gain better understanding on the behavior of the estimators, we next consider their limiting distributions.
\begin{theorem}\label{prop:consistency_and_limiting2}
	Let Assumption \ref{assu:component_rates} hold and assume that,
	\begin{equation}
	\label{eq:key-normality}
	\frac{\sqrt{k_n} \max_\ell \{ g_{n\ell} \}}{c_n} = o(1).
	\end{equation}
	Let \textcolor{black}{$k \in \{1,\ldots,p\}$} be fixed \textcolor{black}{and let $C_H,C_M, \mu_H,\mu_M, \sigma_H$, and $\sigma_M$ be arbitrary constants.}
	\begin{itemize}
		\item[i)] If
		$
		\sqrt{k_n} \left( \hat{\gamma}_H(|\textbf{z}^k|) - C_H \right) \rightsquigarrow \mathcal{N}(\mu_{H}, \sigma_{H}^2),
		$
		then
		\begin{align*}
		\sqrt{k_n} \left( \hat{\gamma}_H(|\hat{\textbf{z}}^k|) - C_H \right) &\rightsquigarrow \mathcal{N}(\mu_{H}, \sigma_{H}^2).
		\end{align*}
		\item[ii)] If $\hat{\gamma}_H(|\textbf{z}^k|) \rightarrow_p C_H$, $\dfrac{\sqrt{k_n}\max_\ell \{ g_{n\ell} \}}{c_n\hat{\gamma}_H(|\textbf{z}^k|)}\rightarrow_p 0$, and
		$
		\sqrt{k_n} \left( \hat{\gamma}_M(|\textbf{z}^k|) - C_M \right) \rightsquigarrow \mathcal{N}(\mu_{M}, \sigma_{M}^2),
		$
		then
		\begin{align*}
		\sqrt{k_n} \left( \hat{\gamma}_M(|\hat{\textbf{z}}^k|) - C_M \right) &\rightsquigarrow \mathcal{N}(\mu_{M}, \sigma_{M}^2).
		\end{align*}
	\end{itemize}
	
\end{theorem}
\textcolor{black}{In the above result, the constants $\mu_H,\mu_M, \sigma_H$, and $\sigma_M$ can be computed explicitly in most cases, their exact values depending on the so-called second order conditions.} For details, we refer to \cite{de2007extreme}. We also remark that in our proof, we could easily replace the convergence rate $\sqrt{k_n}$ with some other rate, or the limiting normal distribution with some other distribution. The underlying reason for the above formulation is that we are not aware of any asymptotic results for extreme \textcolor{black}{value} index estimators where the rate is other than $\sqrt{k_n}$  or where the limiting distribution is not normal.

We end this section by discussing the strictness of the key conditions $\max_{\ell}\{g_{n\ell} \} = o(c_n)$ and $\sqrt{k_n}\max_{\ell}\{g_{n\ell}\} = o(c_n)$. These conditions state that the convergence rate $ c_n $ of the estimated latent sample to the true latent sample must be sufficiently fast compared both to $ \sqrt{k_n} $, the square root of the tail threshold, and to $ \max_\ell \{ g_{n\ell} \} $, the heaviness of the heaviest of the latent components. %EN YMMÄRRÄ (siis Pauliina ei ymmärrä): That is, the extreme value indices are estimable from the estimated latent sample if they are not masked either by the model estimation or by the tails of the components.
Moreover, the rate $ k_n $ can be seen as a type of a tuning parameter. Choosing a faster growing $ k_n $ will make the Hill estimator converge more rapidly, but it will, at the same time, limit the range of distributions whose extreme value indices we can estimate in the first place, and vice versa.

To shed further light on these conditions, we consider an example. %As our interest is in heavy-tailed components, we assume without loss of generality
Assume that there exists at least one latent component belonging to the domain of attraction of the Fr{\'e}chet distribution, i.e., $ \max_\ell \{ g_{n\ell} \} = \mathcal{O}(n^\gamma) $ for some $ \gamma > 0 $ (cf. Lemma \ref{lem:interpretation_g} in the supplementary Appendix \ref{sec:proofs}). Now, letting $ k_n = n^\alpha $ for some $ \alpha > 0 $ and under the standard rate $ c_n = \sqrt{n} $, we end up \textcolor{black}{with} the restriction $ \gamma < \frac{1}{2} (1 - \alpha) $. Hence, putting the tail threshold $ k_n $ sufficiently small, we see that extreme value index estimation is feasible as long as the heaviest Fr{\'e}chet component among the latent variables has its extreme value index smaller than $ 1/2 $, that is, all latent components have finite variance. Heavier components, i.e., ones without second moments, can be captured through estimators which yield faster convergence rates $ c_n $ than the usual $ \sqrt{n} $ for the model estimation. Conversely, if the convergence rate $c_n$ is slower than the usual $\sqrt{n}$ (see, e.g., \cite{lietzen2019}), then $c_n$ might not be sufficient to compensate too heavy tails, and, e.g., assumptions on the existence of higher moments are required.

\section{Example models}\label{sec:two_models}

In this section, we illustrate the applicability of our main results by considering two popular example models: stationary independent component model and stationary second order source separation model. For simplicity, we only consider the Hill estimator, although the following analysis could be easily extended for the moment estimator as well (see Remark \ref{remark:moment-estimator}). Throughout, we assume that the heaviest component has index $\gamma > 0$, often implying that $ \max_\ell \{ g_{n\ell} \} = \mathcal{O}(n^\gamma),$ see the examples below. As the rate $c_n = \sqrt{n}$ is the best possible that one \textcolor{black}{can usually} expect, we also assume $\gamma<\frac12$. This ensures the square integrability of all of our random variables, which is also a minimum requirement for  \eqref{eq:general_model} to hold for the standard estimation procedures in our example models.

Let now \textcolor{black}{$k \in \{1,\ldots,p\}$} be fixed. We illustrate our results in cases where both Theorem~\ref{prop:consistency_and_limiting} and Theorem \ref{prop:consistency_and_limiting2} are applicable. Thus, in order to obtain limiting normality for the Hill estimator $\hat{\gamma}_H(|\textbf{z}^k|)$ based on the true values $|\textbf{z}^k|$, we impose a second order condition for the marginal distribution $F$ of $|z^k|$. The distribution $F$ is called \emph{second order regularly varying} (with index $\gamma$) if there exists a positive or negative function $A$ with the property $\lim_{t\to \infty} A(t) =0$ such that, for all $x>0$,
\begin{equation}
\label{eq:2nd-regularly-varying}
\lim_{t\to\infty} \dfrac{\frac{U(tx)}{U(t)}-x^\gamma}{A(t)} = x^\gamma \frac{x^\rho-1}{\rho},
\end{equation}
holds for some real number $\rho\leq 0$. Here the function $U$ is given by
$$
U = \left(\frac{1}{1-F}\right)^{\leftarrow},
$$
where $^{\leftarrow}$ denotes the left-continuous (pseudo-)inverse function. Then, in the case of independent observations, the limiting normality,
\begin{equation}
\label{eq:Hill-normality}
\sqrt{k_n} \left( \hat{\gamma}_H(|\textbf{z}^k|) - \textcolor{black}{\gamma} \right) \rightsquigarrow \mathcal{N}\left(\frac{\lambda}{1-\rho}, \sigma^2\right),
\end{equation}
holds, provided that $\lim_{n\to \infty} \sqrt{k_n}A\left(\frac{n}{k_n}\right) = \lambda \in \mathbb{R}$. This leads to an upper bound on the rate at which $k_n$ can grow. Similarly, conditions of Theorems \ref{prop:consistency_and_limiting} and  \ref{prop:consistency_and_limiting2} give upper bounds for the rate at which $k_n$ can grow. Thus, we can obtain limiting normality (and consistency) by choosing a not-too-rapidly growing sequence $k_n$, at the cost of a slower rate of convergence. For details on the limiting normality of the Hill estimator in the case of i.i.d. observations, see \cite{de2007extreme}, and in the case of stationary dependent observations, see \cite{Haan-et-al-2016} and the references therein.

\subsection{Independent component model}
\label{subsec:ICA}
In independent component analysis (ICA) the observed \textcolor{black}{$ p $-vectors $ \textbf{x}_1, \ldots , \textbf{x}_n $} are assumed to be a random sample from the independent component (IC) model,
\begin{align}\label{eq:ic_model}
\textbf{x} = \boldsymbol{\Omega} \textbf{z} + \boldsymbol{\mu},
\end{align}
where the latent \textcolor{black}{$ p $-vector $\textbf{z} $} has independent components, $\boldsymbol{\Omega}\in \mathbb{R}^{p\times p}$ is invertible and $\boldsymbol{\mu}\in \mathbb{R}^p$ is a location parameter \citep{hyvarinen2000independent} . The objective in ICA is find an unmixing matrix $\boldsymbol{\Gamma} \in \mathbb{R}^{p \times p}$, such that $\boldsymbol{\Gamma} \textbf{x}$ has independent components. Standard theory then shows that if at most one of the ICs is Gaussian, any such solution coincides with $ \textbf{z} $ up to scaling, order and signs of the components. The scales can be fixed by second order standardisation of $ \textbf{z} $. This guarantees that all the solutions are of the form $ \boldsymbol{\Gamma} = \textbf{P} \textbf{J} \boldsymbol{\Omega}^{-1}$ where $ \textbf{P} \in \mathbb{R}^{p \times p}$ is a permutation matrix and $ \textbf{J} \in \mathbb{R}^{p \times p}$ is a sign-change matrix (diagonal matrix with diagonal entries equal to $ \pm 1 $). In our approach in assessing extreme behaviour, the sign ambiguity is of no concern, as we consider the absolute values of the source components. Moreover, the order of the components is irrelevant, if one is interested in modelling the tail index of the component with the highest risk.

%While this is generally sufficient in standard applications, in extreme value theory the orientations (signs) of the sources are of crucial importance. This is one of the reasons blah and as discussed in Section \ref{sec:introduction} we take a conservative approach and estimate not the tails of the source components but their absolute values. %A second way around the issue would be to require all sources to have marginally symmetric distributions, making the two tails identical. This is, however, impractical for two reasons: first, in the presence of symmetric sources our chosen approach yields results equivalent to those obtained without resorting to absolute values, making the approach based on symmetric sources strictly less general, and second, the assumption on symmetric sources seems questionable in applications such as finance where extreme losses rarely behave identically to extreme wins.

Numerous estimators $ \hat{\boldsymbol{\Gamma}} $ of the unmixing matrix have been proposed and under suitable assumptions and standardisations, most estimators converge to $ \boldsymbol{\Omega}^{-1} $ at some rate $ c_n $,
\begin{align}\label{eq:root_n_ica}
c_n \left( \hat{\boldsymbol{\Gamma}} \boldsymbol{\Omega} - \textbf{I}_p \right) = \mathcal{O}_p(1).
\end{align}
Typically, in the context of i.i.d. observations, we have $ c_n = \sqrt{n} $. See \cite{miettinen2015fourth} for several examples including FastICA \citep{hyvarinen1999fast}, fourth order blind identification (FOBI) \citep{cardoso1989source}, and joint approximate diagonalization of eigenmatrices (JADE) \citep{cardoso1993blind}. Also the ICA-estimators based on the simultaneous diagonalization of two symmetrized scatter matrices \citep{oja2006scatter, nordhausen2008robust} can be shown to have the rate $ \sqrt{n} $, assuming that the applied symmetrized scatter matrices have the same convergence rate.

\textcolor{black}{Assuming} that $ \hat{\boldsymbol{\Gamma}} $ is of the form \eqref{eq:root_n_ica}, the estimated latent vectors can be written as
\[
\hat{\textbf{z}}_i = \hat{\boldsymbol{\Gamma}} \left( \textbf{x}_i - \bar{\textbf{x}} \right) = \hat{\boldsymbol{\Gamma}} \boldsymbol{\Omega} \left( \textbf{z}_i - \bar{\textbf{z}} \right) = \textbf{z}_i + \left( \hat{\boldsymbol{\Gamma}} \boldsymbol{\Omega} - \textbf{I}_p \right) \textbf{z}_i - \hat{\boldsymbol{\Gamma}} \boldsymbol{\Omega} \bar{\textbf{z}}.
\]
Writing now $ \hat{\textbf{H}} := \hat{\boldsymbol{\Gamma}} \boldsymbol{\Omega} - \textbf{I}_p $ and $ \hat{\textbf{r}} := - \hat{\boldsymbol{\Gamma}} \boldsymbol{\Omega} \bar{\textbf{z}} $, we observe that we have arrived to the form~\eqref{eq:general_model}. \textcolor{black}{By the assumption that} $ \max_\ell \{ g_{n\ell} \} = \mathcal{O}(n^\gamma) $ with $\gamma<\frac12$, we observe that \eqref{eq:key-consistency} is automatically valid, and \eqref{eq:key-normality} is valid for suitably chosen sequence $k_n$. We stress that $\gamma<\frac12$ guarantees the existence of second moments, while usually standard \textcolor{black}{ICA} methods operate on higher-order information making even stronger moment assumptions. For example, the $ \sqrt{n} $-consistency for FOBI requires the existence of finite eighth moments of the latent variables  \citep{ilmonen2010characteristics}. %This would require $\gamma<\frac18$, which gives more freedom for the choice of $k_n$.
By using squared FastICA with the hyperbolic tangent \citep{miettinen2017squared} as the ICA-estimator, one can reduce the order of the required moments to four. Finally, we stress that under independent observations drawn from a second order regular varying heavy tailed distribution, the Hill estimator $\hat{\gamma}_H(|\textbf{z}^k|)$ is consistent and asymptotically normal (see \cite{de2007extreme}), and while different technical assumptions are required for the classical ICA-estimators, none of them interfere with our assumption that guarantees the consistency and limiting normality of the Hill estimator. %Indeed, by independence of the observations the latter requires only that $|z^k|$ is heavy tailed.
Thus, as a conclusion, we can safely apply Theorem \ref{prop:consistency_and_limiting} and Theorem \ref{prop:consistency_and_limiting2}.
\begin{remark}
	\label{remark:moment-estimator}
	In the above discussions we have considered only the Hill estimator. However, applying Theorem \ref{prop:consistency_and_limiting} or Theorem \ref{prop:consistency_and_limiting2} for the moment estimator in the ICA-context is straightforward. Indeed, under second order regularly varying tails and independence, the Hill estimator always converges to \textcolor{black}{$\max(0,\gamma)$}, and the moment estimator is both consistent and asymptotically normal. Thus it suffices to check the extra condition $\frac{\sqrt{k_n}\max_\ell \{ g_{n\ell} \}}{c_n\hat{\gamma}_H(|\textbf{z}^k|)}\rightarrow_p 0$. However, even if \textcolor{black}{$\gamma \leq 0$} we have $\sqrt{k_n}\hat{\gamma}_H(|\textbf{z}^k|) \rightsquigarrow \mathcal{N}(\mu_{H}, \sigma_{H}^2)$, and thus it suffices to choose the sequence $k_n$ such that $\frac{k_n\max_\ell \{ g_{n\ell} \}}{c_n} =o(1)$.
\end{remark}

\subsection{Second order source separation model}\label{subsec:BSS}
Our second example moves to the realm of signal processing and blind source separation (BSS). Like the IC model, also the second order BSS model is linear and based on the general location-scatter model. In the model, the observed run \textcolor{black}{$ \textbf{x}_1, \ldots , \textbf{x}_n $} of a stationary $ p $-variate time series is assumed to have the instantaneous latent representation,
\begin{align}\label{eq:bss_model}
\textbf{x}_i = \boldsymbol{\Omega} \textbf{z}_i + \boldsymbol{\mu}, \quad \textcolor{black}{i \in \{1, \ldots , n\}},
\end{align}
where the latent $ p $-variate time series $ \textbf{z}_i $ is stationary  and has standardized uncorrelated components, and $\boldsymbol{\Omega}\in \mathbb{R}^{p\times p}$ is of full rank. The location $\boldsymbol{\mu}\in \mathbb{R}^p$ is (by stationarity) trivial to estimate by using a standard average estimator, which provides a consistent estimator if the system is ergodic. Thus, for the sake of simplicity, it will be omitted in the following. %Assuming the independence of the component series of $ \textbf{z}_i $ is not standard (and indeed, the theory works just as well with plain uncorrelatedness), but we include it to preserve our central idea of estimating from the data a set of multiple \textit{independent} causes that together constitute the observed risk.
Note also that the non-identifiability of signs and order holds in the BSS model as well. However, for our purposes this does not matter due to the reasons explained in Subsection~\ref{subsec:ICA}.

One standard approach to estimate $\textbf{z}_i$ is algorithm for multiple unknown signals extraction (AMUSE) \citep{tong-et-al} where the autocovariance matrices $ \boldsymbol{\Sigma}_\tau(\textbf{x}_i) = \mathbb{E} ( \textbf{x}_i \textbf{x}_{i + \tau}^\top ) $ for $\tau \in \{0,\tau_0\}$ are diagonalised simultaneously.
An extension of AMUSE that is less sensitive \textcolor{black}{to} the choice of $\tau_0$ is the second order blind identification (SOBI) \citep{belouchrani1997blind} algorithm where the autocovariance matrices $ \boldsymbol{\Sigma}_\tau(\textbf{x}_i) = \mathbb{E} ( \textbf{x}_i \textbf{x}_{i + \tau}^\top ) $ over a chosen set of lags $ \mathcal{T} = \{ \tau_1, \ldots , \tau_{|\mathcal{T}|} \} $ are jointly diagonalized. % to estimate an inverse $\hat{\boldsymbol{\Gamma}}$ for $\boldsymbol{\Omega}$.
As in the IC-model, one would expect that the algorithm provides a
consistent estimator $ \hat{\boldsymbol{\Gamma}} $ with some rate $c_n$:
\begin{equation}
\label{eq:BSS-convergence}
c_n \left( \hat{\boldsymbol{\Gamma}} \boldsymbol{\Omega} - \textbf{I}_p \right) = \mathcal{O}_p(1).
\end{equation}
It turns out that \eqref{eq:BSS-convergence} holds true whenever
\begin{equation}
\label{eq:key-covariance}
c_n\left(\hat{\boldsymbol{\Sigma}}_\tau(\textbf{x}_i) - \boldsymbol{\Sigma}_\tau(\textbf{x}_i)\right) = \mathcal{O}_p(1), \quad \tau \in \mathcal{T},
\end{equation}
where $\hat{\boldsymbol{\Sigma}}_\tau(\textbf{x}_i)$ denotes the estimator of the autocovariance matrix $\boldsymbol{\Sigma}_\tau(\textbf{x}_i)$. The fact that \eqref{eq:key-covariance} implies \eqref{eq:BSS-convergence} is proved in the case of complex valued AMUSE in \cite{lietzen2019} with general rate $c_n$, and in the case of real valued SOBI (and its variants) in
\cite{miettinen2016separation} for the rate $c_n = \sqrt{n}$. It is also straightforward to check that the arguments of \cite{miettinen2016separation} apply with arbitrary rate function $c_n$. For examples with general rate $c_n$ instead of the standard $\sqrt{n}$, we refer to \cite{lietzen2019}.

\textcolor{black}{Equation \eqref{eq:key-covariance} is} the first key assumption on the rate of convergence for the autocovariance estimators, which on the other hand gives us our speed $c_n$. If $c_n$ is non-standard, \eqref{eq:key-consistency} gives us also the restriction $n^\gamma = o(c_n)$\textcolor{black}{, limiting the} possible values of $\gamma$. This can be seen as an interchange between moment assumptions and the speed at which the autocovariance estimators converge, as higher moments are required if the estimators converge slowly.

In order to make Theorem \ref{prop:consistency_and_limiting2} applicable, we also require that the Hill estimator 		$\hat{\gamma}_H(|\textbf{z}^k|)
$ satisfies limiting normality \eqref{eq:Hill-normality}. Compared to independent observations, the problem is much more subtle in the case of dependent sequences and one needs to pose extra assumptions  in addition to the second order regularly varying condition \eqref{eq:2nd-regularly-varying}. The extra assumptions are, roughly speaking, conditions that ensure the dependence to be weak enough so that the series ''behaves'' similarly as a series of independent observations. The precise definition of weak dependence or asymptotic independence varies in the literature. Usually asymptotic independence is encoded to mixing-conditions (for different notions of mixing-conditions and their relations, see the survey in \cite{Bradley}). It is known (see, e.g., \cite{Drees2000,Drees2003}) that \eqref{eq:Hill-normality} holds provided that \textcolor{black}{$|\textbf{z}^k|$} forms a $\beta$-mixing
stationary sequence such that some minor additional regularity conditions are met (see, e.g., conditions (a)-(c) of \cite{Haan-et-al-2016}). In particular, all these conditions are satisfied for the following sequences:
\begin{itemize}
	\item $m$-dependent process and AR(1)-process \citep{Drees2003, Rootzen1995, Rootzen2009},
	\item AR(p)-processes and MA($\infty$)-processes (with suitable assumptions on the coefficients) \citep{Drees2002,Resnick-Starica1997},
	\item MA(q)-processes \citep{Drees2002,Hsing,Rootzen1995,Rootzen2009},
	\item ARCH(1)-processes \citep{Drees2002,Drees2003},
	\item GARCH-processes \citep{Drees2000,Starica1999}.
\end{itemize}
We emphasize that the above examples form a very large and applicable class of processes. For details and more information on the above examples, see also \cite{Haan-et-al-2016}.

We now turn back to \textcolor{black}{extreme value} index estimation under the BSS model. In order to apply Theorem \ref{prop:consistency_and_limiting} or Theorem \ref{prop:consistency_and_limiting2}, it suffices to make sure that, for a given heavy tailed component \textcolor{black}{$|\textbf{z}^k|$}, the above mentioned conditions guaranteeing the limiting normality \eqref{eq:Hill-normality} for the Hill estimator $\hat{\gamma}_H(|\textbf{z}^k|)$ are satisfied. At the same time, one needs that \eqref{eq:key-covariance} holds, with some rate $c_n$ satisfying $n^\gamma = o(c_n)$. After that, it remains to choose $k_n$ not increasing too rapidly so that \eqref{eq:key-normality} holds as well. We next explore the connection between the given assumptions. We first observe that conditions required to ensure \eqref{eq:Hill-normality} are solely on the dependence structure and distribution of the given component \textcolor{black}{$|\textbf{z}^k|$} of interest. At the same time,  condition \eqref{eq:key-covariance} considers the rate of convergence of autocovariance estimators for all components simultaneously. %Thus, for example, if the components are independent, the rate $c_n$ in \eqref{eq:key-covariance} can be caused by an independent component $z^j, j\neq k$. On the other hand, while conditions ensuring \eqref{eq:Hill-normality} and convergence related to other components in \eqref{eq:key-covariance} are not intimately connected, we stress that \eqref{eq:key-covariance} implies certain rate of convergence for the autocovariance estimators of the component $z^k$ of interest as well.
Note that $\beta$-mixing and assumptions (a)-(c) of \cite{Haan-et-al-2016} do not imply convergence of the autocovariance estimators (as the conditions do not even require existence of second moments). Conversely, convergence of the autocovariance estimators is related to the so-called $\rho$-mixing (see \cite{Bradley} for precise definition) which does not imply $\beta$-mixing.  Thus, even the convergence rate of the autocovariance estimator of the component \textcolor{black}{$|\textbf{z}^k|$} does not provide any information regarding the validity of conditions implying limiting normality \eqref{eq:Hill-normality} for the Hill estimator. This means that the assumptions do not contradict, and also that in practice, one has to verify \eqref{eq:key-covariance} and the limiting normality of the Hill estimator separately.
\begin{remark}
	If all components are $\rho$-mixing, then the slowest decay of the $\rho$-mixing coefficients gives us an upper bound for $c_n$. Moreover, if one poses a stronger mode of mixing, $\phi$-mixing, for the sequence \textcolor{black}{$|\textbf{z}^k|$}, then \citep[p.112]{Bradley} the sequence is also both $\beta$- and $\rho$-mixing. For example, this is the case for $m$-dependent processes and GARCH-processes.
\end{remark}

\section{Simulations}\label{sec:examples}

In this section, we illustrate the tail index estimation under the second order source separation model of Section \ref{subsec:BSS} through a simulation study. Appendix \ref{sec:appendix_simulation} in the supplementary material presents a similar study for the independent component model in Section \ref{subsec:ICA}, with largely the same conclusions.

In this simulation study, we consider the $\mathbb{R}^3$-process $\textbf{z}$, where the components are $n$-length realizations, i.e., time series, of the independent stochastic processes in $ \tilde{\textbf{z}} = ( \texttt{ARCH}(1), \texttt{D}^{(1)}, \texttt{D}^{(2)})^\top. $ The first component of $ \tilde{\textbf{z}} $ is an ARCH(1)-process with the parameter vector $ (\alpha_0, \alpha_1) = (1/4, ({2^3 \sqrt{2/\pi}})^{-2/5}) $.
At time $t$, the second and third components are defined as $ \texttt{D}^{(1)}_t  = \left( B_{t+1}^{(1)} - B_t^{(1)}\right)^2 - 1 $ and $ \texttt{D}^{(2)}_t  = \left( B_{t+1}^{(2)} - B_t^{(2)}\right)^2 - 1 $, where $B^{(1)}$ denotes a fractional Brownian motion (fBm) with Hurst parameter 3/4 and  $B^{(2)}$ denotes a fBm with Hurst parameter 4/5, such that $B^{(1)}$ are $B^{(2)}$ are mutually independent. For a comprehensive study on fBm, see, e.g., \cite{nualart2006fractional}. Out of the three components, \textcolor{black}{the ARCH(1) process} has the largest theoretical \textcolor{black}{extreme value} index~1/5. We considered the sample sizes $n \in \{  300, 10^3, 10^4, 10^5, 10^6, 10^7 \}$ and the threshold sequence $k_n$ was chosen to be $k_n = \lfloor n^{1/4} \rfloor$. For each sample size, the simulation was iterated 2000 times.

As a preliminary step, the simulated observations \textcolor{black}{$\tilde{\textbf{z}}_i $} were centered. Here, the centered observations are denoted as \textcolor{black}{${\textbf{z}}_i$}. In every iteration $h \in \{1,\ldots 2000\}$, we applied, for all $  i\in \{1,\ldots,n\} $, the linear transformation $ \textbf{x}_i = \boldsymbol{\Omega}_h {\textbf{z}}_i $, where the elements of the $\mathbb{R}^{3\times 3}$-matrix \textcolor{black}{$\boldsymbol{\Omega}_h$} were simulated independently, and separately in every iteration, from the univariate uniform distribution $\texttt{unif}(-100,100)$. We then applied the AMUSE unmixing procedure with lag $\tau =1$ to the mixed time series, using the implementation contained in the R-package JADE \citep{Rjade}. The existence of the limiting distribution of the AMUSE unmixing estimator requires finite fourth moments. Note that the ARCH(1) parameters $\alpha_0, \alpha_1$ are chosen such that the fourth moments exist for all components. We denote the absolute values of the AMUSE unmixed time series and the absolute values of the \textcolor{black}{original} centered time series as $| \hat{\textbf{z}} |$ and $| \textbf{z} |$, respectively.

\textcolor{black}{Now}, we have $\max_\ell ({g_{n\ell}}) = n^{-1/5}$, which corresponds to the ARCH(1) process. The $\texttt{D}^{(2)}$ process in the third component has the slowest rate of converence, giving $c_n = n^{2/5 }$, see  \cite{lietzen2019}. Hereby, under our choice of $k_n = \lfloor n^{1/4} \rfloor $, we have that the assumptions required by \textcolor{black}{Theorems \ref{prop:consistency_and_limiting} and \ref{prop:consistency_and_limiting2}} hold and, hence, for large sample sizes, the \textcolor{black}{extreme value} index estimates calculated from $|\hat{\textbf{z}}|$ and $|{\textbf{z}}|$ are expected to be close to each other.

We estimated the \textcolor{black}{extreme value} indices for every component from \textcolor{black}{both} $|\hat{\textbf{z}}|$ and $|{\textbf{z}}|$, using both the Hill estimator and the moment estimator. Note that both estimators produce three \textcolor{black}{extreme value} index estimates, one for each component. \textcolor{black}{To capture the ARCH(1) component, we collected, in every simulation iteration,} the largest of the three estimates, denoted in the following by $ \hat{\gamma}(|\hat{\textbf{z}}|) $ and $ \hat{\gamma}(|{\textbf{z}}|) $ \textcolor{black}{(this induces a slight bias to the results which is, however, rendered negligible with increasing~$ n $)}. The histograms of $ \hat{\gamma}(|\hat{\textbf{z}}|) $ and $ \hat{\gamma}(|{\textbf{z}}|) $ for sample sizes $n=300,10^3,10^4$ are shown in Figure~\ref{fig:simu_3}, where the \textcolor{black}{extreme value} indices estimated from $|\hat{\textbf{z}}|$ correspond to light blue colour, and the \textcolor{black}{extreme value} index estimates calculated from the original $|{\textbf{z}}|$ correspond to light red colour. \textcolor{black}{Dark blue colour is used for the parts of the histograms} that overlap and the dashed yellow vertical line represents the theoretical extreme value index value $\gamma = 1/5$. Values smaller than $-2$ are omitted from the figure\textcolor{black}{;} a total of 21 moment estimator estimates were smaller than~$-2$.

\begin{figure}[htp]
	\includegraphics[width=0.9\textwidth]{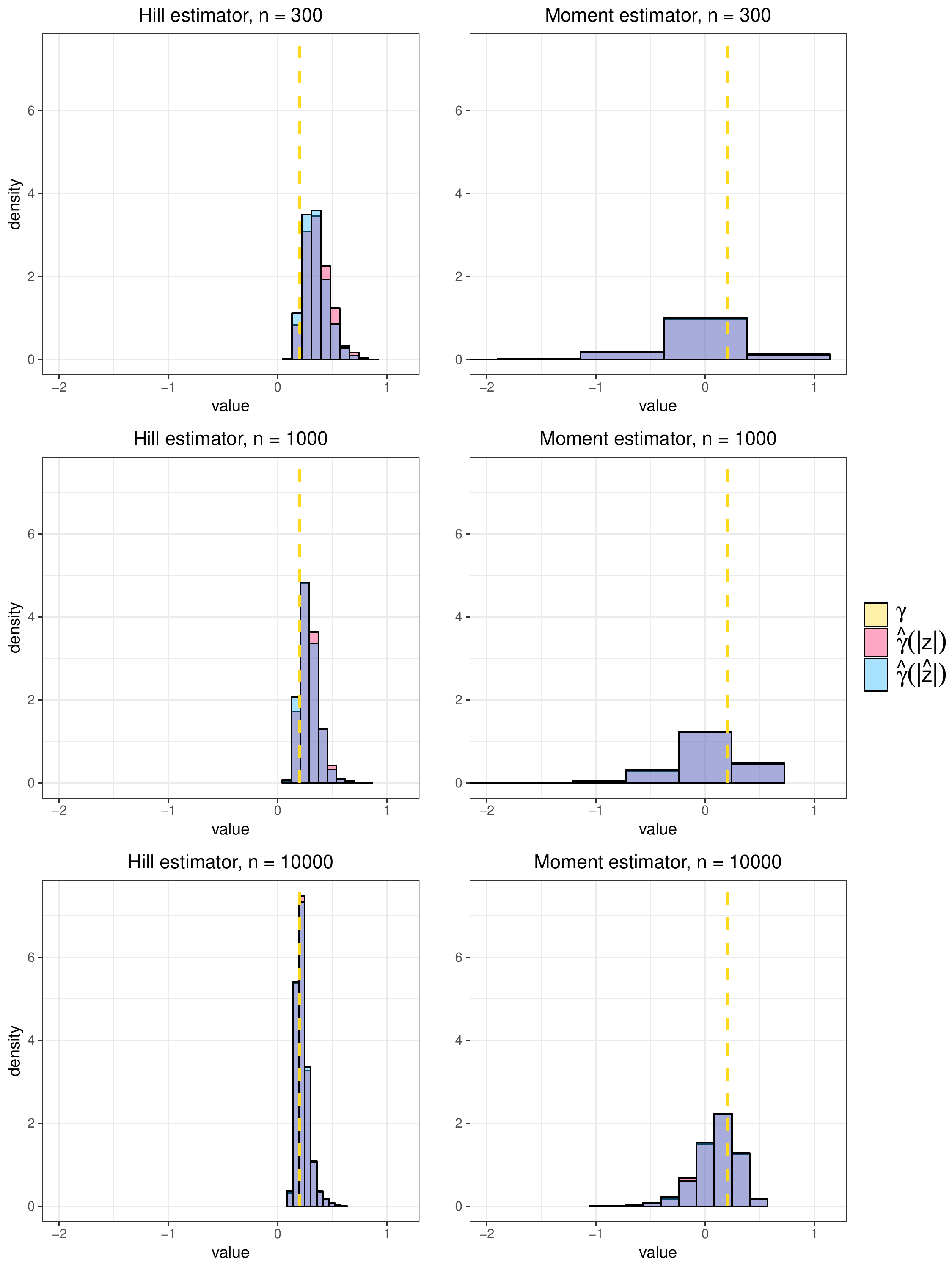}
	\caption{
		Histograms of $ \hat{\gamma}(|{\textbf{z}}|) $ (light red) and $ \hat{\gamma}(|\hat{\textbf{z}}|) $ (light blue) in the simulation study with sample sizes 300, 1000 and 10 000. The dashed yellow vertical line is the theoretical \textcolor{black}{extreme value} index $\gamma = 1/5$. The dark blue color in the histograms represents the area, where the two histograms overlap.}
	\label{fig:simu_3}
\end{figure}

\textcolor{black}{In Figure \ref{fig:simu_3}, already for the small sample size $ n = 300 $, the two histograms overlap significantly. Moreover, starting from $ n = 1000 $, the histograms are basically identical, showing that, as predicted by the theory, the effect of the BSS-step on the estimation of the extreme value indices is almost negligible. When comparing the Hill estimator and the moment estimator, Figure \ref{fig:simu_3} indicates that the variance of the moment estimator is larger, when compared to the Hill estimator. In addition, the bias of the Hill estimator is visible in the histograms, see \cite{de2007extreme}, and seems to decrease as the sample size increases. The histograms corresponding to the sample sizes $10^5,10^6$ and $10^7$ have been omitted here, as they introduce no new information to the simulation study.}

\begin{figure}[htp]
	\includegraphics[width=1\textwidth]{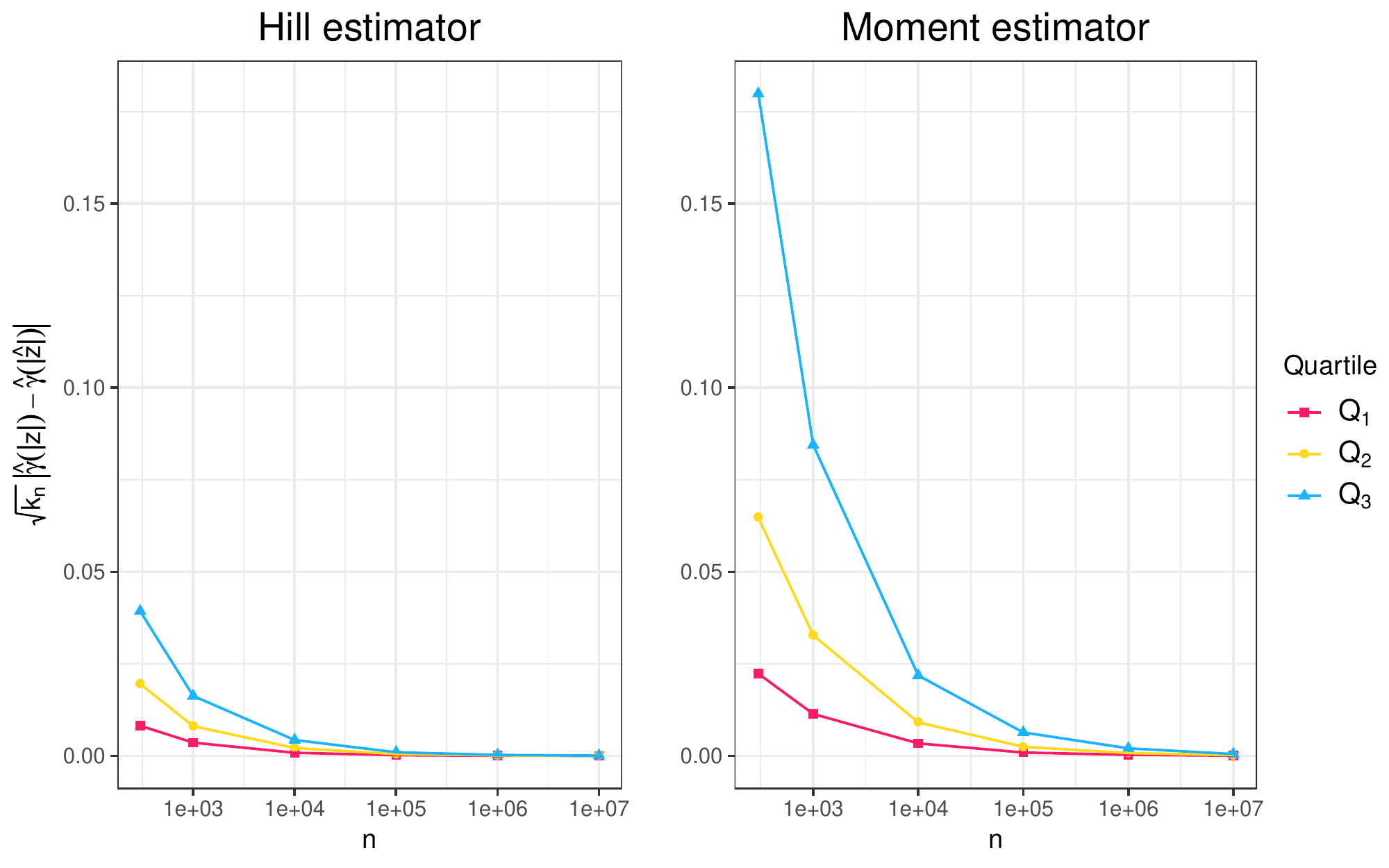}
	\caption{The quartiles of $\sqrt{k_n}|\hat{\gamma}(|\textbf{z}|) -\hat{\gamma}(|\hat{\textbf{z}}|) |$, for the Hill estimator and the  moment estimator in the simulation study.}
	\label{fig:simu_4}
\end{figure}

Figure \ref{fig:simu_4} illustrates the absolute differences, scaled with $\sqrt{k_n}$, between the estimates calculated from $|\textbf{z}|$ and $|\hat{\textbf{z}}|$. The red and blue curves represent the first and third empirical quartiles of the absolute differences, respectively, \textcolor{black}{and the yellow curve is the corresponding sample median curve.} \textcolor{black}{The differences can be seen to converge to zero for both estimators, but the moment estimator requires larger sample sizes for this. That is, the quartile $Q_3$ for the Hill estimator is close to zero already with $ n = 10^5$ and, conversely, the moment estimator quartile $Q_3$ requires samples of size $ n = 10^7 $ for achieving the same magnitude.}

\section{Real data example}\label{sec:real_data}

Heavy-tailed distributions are encountered frequently in the context of financial instruments \citep{rachev2003handbook}. Here, we consider extreme value index estimation for a four-dimensional financial time series downloaded from Yahoo Finance. The data consist of the daily log-returns of the S\&P500 index and the stock prices of CISCO Systems, Intel Corporation and Sprint Corporation in the period of January 3rd, 1991 -- September 12th, 2019. The observations were further standardized to have unit variance, which can be done without loss of generality as both our extreme value index estimators are scale invariant. A subset of the data from a shorter period of time was used already in \cite{fan2008modelling} in the context of multivariate volatility modeling, which inspired us to choose the same data set.

The full four-variate series is visualized in \textcolor{black}{Figure \ref{fig:example_1} in the supplementary Appendix \ref{sec:appendixC}.} Volatility spikes that span most of the series occur around the years 2002 and 2008, caused by the stock market downturn of 2002 and the financial crisis of 2007--2008, respectively. Especially the latter time period stands out also in Figure \ref{fig:example_2}, where we have estimated the extreme value indices of the individual series. The estimation in Figure \ref{fig:example_2} was conducted by moving a window of length 60 days through each univariate series and estimating the extreme value index of each window with the Hill estimator with the tail length $ k_n = k = 16 $. The $ x $-axis values in the plot correspond to the middle days (30th days) of the windows. Contrary to the approach in Section \ref{sec:general_framework}, we estimated the extreme value indices not from the absolute values of the series, but \textcolor{black}{separately for} both the left and  \textcolor{black}{the right tail of each of the series.} That is, for each of the four time series \textcolor{black}{in Figure \ref{fig:example_1}}, we obtain two sequences of extreme value index estimates, always plotted with the same colours in Figure \ref{fig:example_2}. This approach was taken to assess the behaviour of both negative and positive returns separately, \textcolor{black}{in order to perform a more subtle analysis}.  Based on \textcolor{black}{Figures  \ref{fig:example_2} and \ref{fig:example_1}}, it seems reasonable to assume that among the four series there is an underlying latent factor (``financial crisis series'') which contributes risk to all four series around the times of the previous two crises.

\begin{figure}[htp]
	\includegraphics[width=1\textwidth]{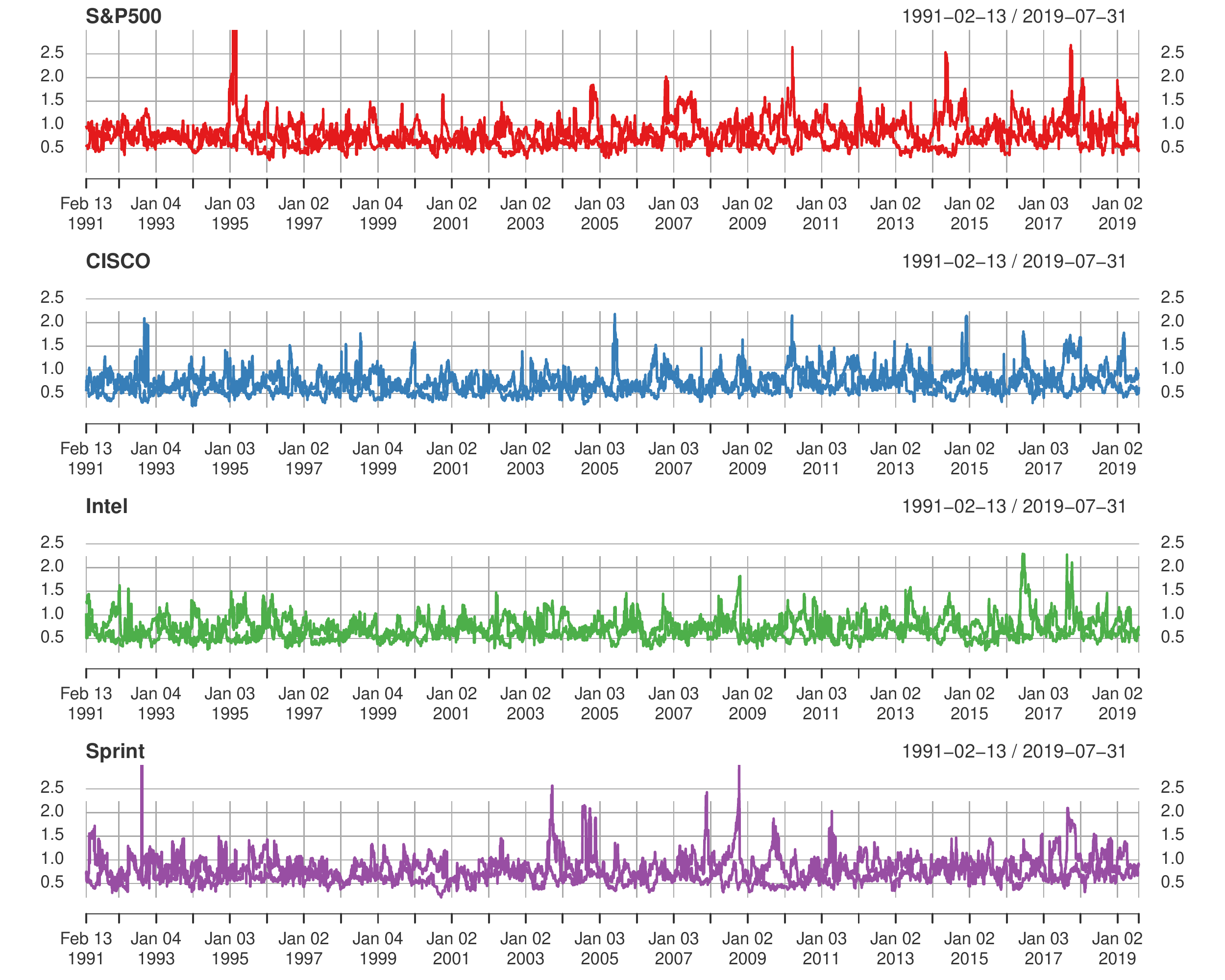}
	\caption{The extreme value indices of the observed series $ \textbf{x}_i $ estimated with a rolling window of length 60 days. The two series in each plot correspond to the extreme value index estimates of the left and right tails of the corresponding series. Hill estimator with the tail length $ k_n = 16 $ was used. The $ x $-axis in the plot denotes the middle (30th) days of the windows.}
	\label{fig:example_2}
\end{figure}

To explore this, we estimate latent factors using generalized SOBI \citep{miettinen2019extracting}, an extension of the SOBI method which uses both serial correlation and volatility information in estimating the latent series. Denoting the original \textcolor{black}{four-variate series at time $ i $ by $ \textbf{x}_i $}, the estimates of the centered latent series are given by $ \hat{\textbf{z}}_i = \hat{\boldsymbol{\Gamma}} (\textbf{x}_i - \bar{\textbf{x}}) $ where $ \hat{\boldsymbol{\Gamma}} \in \mathbb{R}^{4 \times 4} $ is the unmixing matrix estimate given by generalized SOBI. The estimates are shown in Figure \ref{fig:example_3} in the supplementary Appendix \ref{sec:appendixC} and indeed hint that the risk on certain periods is driven by individual latent factors. E.g., the majority of the volatility associated with the 2007-2008 financial crisis has concentrated in the fourth latent series.

To get a clearer view, Figure \ref{fig:example_4} shows the extreme value index estimates of the four latent series, obtained using the same rolling window approach as used in Figure~\ref{fig:example_2}. The most prominent feature in Figure \ref{fig:example_4} is the spike around \textcolor{black}{year 2002} in one of the extreme value indices of the first series, indicating a period of large risk. Several other spikes are also visible, most notably \textcolor{black}{in one of the indices of the fourth \textcolor{black}{latent} series during late 2002. Thus, we \textcolor{black}{infer} that the 2002 crisis was driven by two separate sources of risk.}%, the latter of which later contributed to the crisis of 2007--2008.

\begin{figure}[htp]
	\includegraphics[width=1\textwidth]{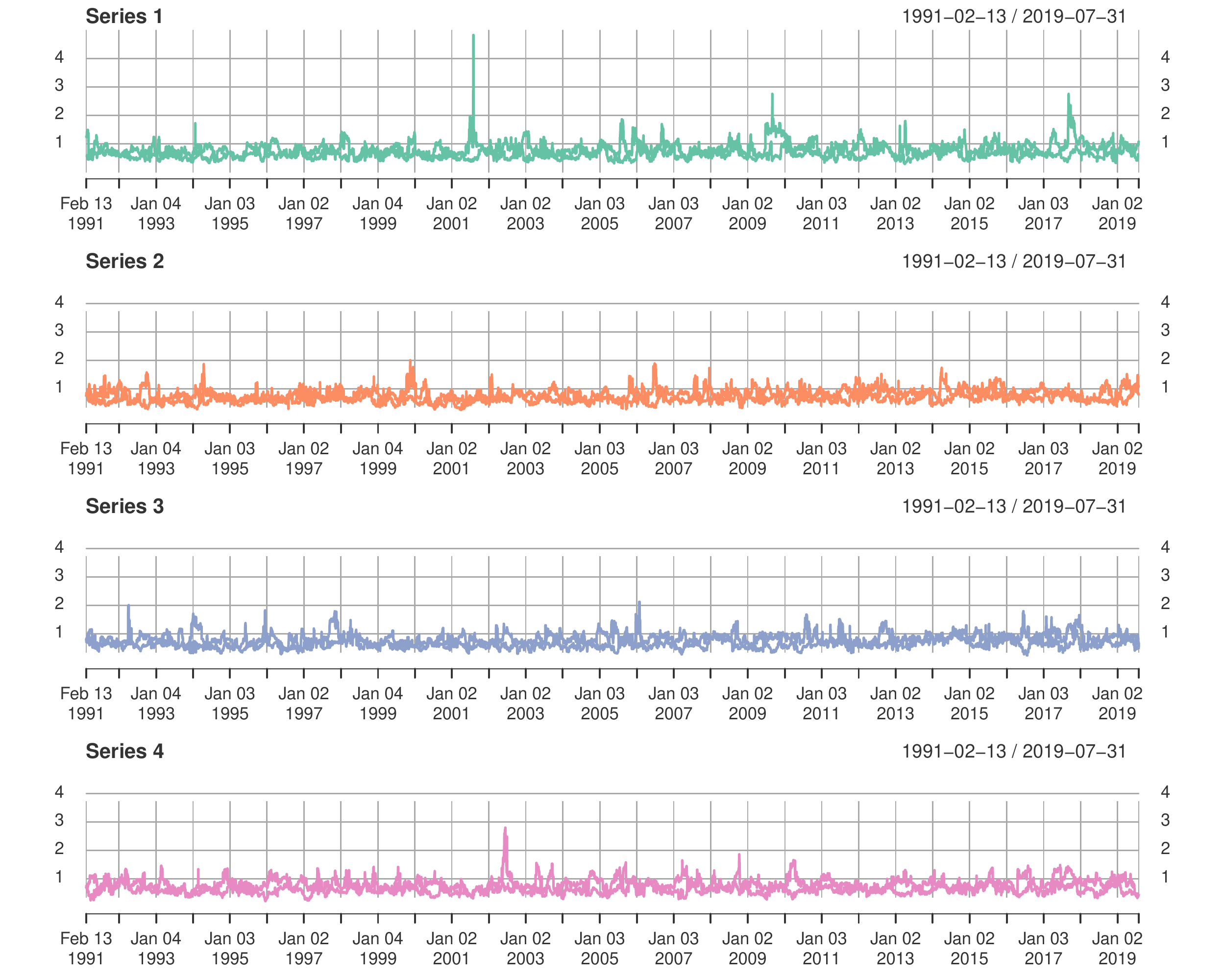}
	\caption{The extreme value indices of the latent series $ \textbf{z}_i $ estimated with a rolling window of length 60 days. The two series in each plot correspond to the extreme value index estimates of the left and right tails of the corresponding series. Hill estimator with the tail length $ k_n = 16 $ was used. The $ x $-axis in the plot denotes the middle (30th) days of the windows.}
	\label{fig:example_4}
\end{figure}

Finally, we study the connection between the factors and the observed series. The inverse transformation from the latent series to the observed ones is $ \hat{\textbf{x}}_i = \hat{\boldsymbol{\Gamma}}^{-1} \textbf{z}_i + \bar{\textbf{x}} $ where
\[
\hat{\boldsymbol{\Gamma}}^{-1} = \begin{pmatrix}
0.54 & 0.16 & 0.20 & 0.80 \\
0.24 & 0.87 & 0.28 & 0.34 \\
0.28 & 0.16 & 0.89 & 0.32 \\
-0.51 & 0.04 & 0.10 & 0.85
\end{pmatrix},
\]
contains the loadings of the \textcolor{black}{latent sources} for each of the observed series. The loadings reveal, for example, that both the first and fourth latent series contribute (absolutely) most to the first and the fourth original time series. \textcolor{black}{More specifically}, the fourth latent process is the most important \textcolor{black}{(loadings 0.80 and 0.85)}, and the first latent process the second most important \textcolor{black}{(loadings 0.54 and -0.51)} in explaining the behavior of the  log-returns of S\&P500 and Sprint. We conclude that, out of the four observed series, the financial crises affected S\&P500 and Sprint the most, and had a significantly smaller impact on CISCO and Intel.

\section{Conclusion}
\label{sec:conc}

We studied the effect of a preliminary latent variable extraction on the estimation of the extreme value indices of the latent independent components. This approach to multivariate extreme value analysis is highly practical in the sense that it reduces the problem into several univariate extreme value problems, allowing the use of the standard extreme value machinery. Moreover, our asymptotic analysis revealed that, under reasonably mild conditions, the consistency and limiting normality of the Hill estimator and the moment estimator are preserved in this construction.

A natural question to pursue in the future is whether the conditions in Theorems \ref{prop:consistency_and_limiting} and \ref{prop:consistency_and_limiting2} can be weakened (we only showed that they are sufficient). Some preliminary simulation (not shown here) indicates that this might indeed be the case. Moreover, the current work can likely be used to simplify the task of deriving similar results for other suitable estimators besides the Hill estimator and the moment estimator. This is because the perturbation bounds for tail observations given in Appendix \ref{subsec:aux-lemmas} are not tied to any particular extreme value index estimator (indeed, they concern the latent variable estimation part of the model). As such, one only needs to derive the analogues of Appendix \ref{subsec:main-proofs} (perturbation bounds for the actual extreme value index estimation step) for the new methods.

%\bigskip
%\begin{center}
%{\large\bf SUPPLEMENTARY MATERIAL}
%\end{center}
%
%\begin{description}
%
%\item[Appendix:] An appendix containing the proofs of the technical results, an additional simulation study and additional figures for the data example in Section \ref{sec:real_data}. (.pdf file)
%
%\item[R-code:] The code for replicating the simulations and the data example. (.zip file)
%
%\end{description}
%

\appendix

\section{Proofs}
\label{sec:proofs}
Section \ref{sec:proofs} of the appendix is devoted to the proofs of the technical results. We have gathered auxiliary technical lemmas into Subsection \ref{subsec:aux-lemmas} and Subsection \ref{subsec:main-proofs} contains to the proofs of our main theorems.
\subsection{Auxiliary lemmas}
\label{subsec:aux-lemmas}
The main objective in this subsection is to establish the rate at which the quantity
\[
| |\hat{\textbf{z}}^k|_{(n-m,n)} / |\textbf{z}^k|_{(n-m,n)} - 1 |
\]
vanishes. We begin with the next result that allows us to consider the component with the heaviest tail as the conservative bound for the error. This translates into $\max_l \{g_{nl}\}$ on our main theorems.

\begin{lemma}\label{lem:interpretation_g}
	Let $ F_0  \in G_{\gamma_0}, F_{1} \in G_{\gamma_1}, F_{2} \in G_{\gamma_2}, F_{3} \in G_{\gamma_3}$ be distributions such that,
	\[
	\gamma_0 > \gamma_1 > \gamma_2 = 0 > \gamma_3.
	\]
	For $k=0,1,2,3$, put $g_{nk} =\max \{ a_{nk}, b_{nk} \}$, where $a_{nk},b_{nk}$ are the normalising sequences such that $\frac{y^k_{(n,n)}-b_{nk}}{a_{nk}} \rightsquigarrow G_\gamma$, where $y^k$ follows $F_k$.
	Then
	\[
	\frac{g_{nk}}{g_{n0}} = o(1), \quad k = 1, 2, 3.
	\]
\end{lemma}
\begin{proof}
	Note first that since $ \gamma_0 > 0 $, the distribution $ F_{0} $ is heavy tailed and belongs to the domain of attraction of the Fr{\'e}chet distribution. Thus, by \cite[Section 3.4]{embrechts2013modelling}, $ g_{n0} = a_{n0} = n^{\gamma_0} L_0(n)$ where $ L_0 $ is a slowly varying function. Similarly $ g_{n1} = n^{\gamma_1} L_1(n)$ for some slowly varying function $ L_1 $ and we have the claim for the value $ k = 1 $, that is,
	\[
	\frac{g_{n1}}{g_{n0}} = n^{\gamma_1 - \gamma_0} \frac{L_1(n)}{L_0(n)} = o(1).
	\]
	Similarly, the distribution $ F_{3} $ is light tailed and belongs to the domain of attraction of the Weibull distribution. As such, by \cite[Section 3.4]{embrechts2013modelling}, we have $ g_{n3} = \max \{ n^{\gamma_3} L_3(n), d \} $ for some slowly-varying function $ L_3 $ and constant $ d $. Since $ \gamma_3 < 0 $, we have $ g_{n3} = \mathcal{O}(1)$ and the claim for $ k = 3 $ follows from
	\[
	\frac{g_{n3}}{g_{n0}} = \frac{\mathcal{O}(1)}{n^{\gamma_0} L_0(n)} = o(1).
	\]
	It remains to prove the case $k=2$ that corresponds to the border case $\gamma_2=0$. Now $F_2$
	belongs to the domain of attraction of the Gumbel distribution and, by \cite[Section 3.4]{embrechts2013modelling}, we have $ g_{n2} = \max \{ a(b_n), b_n \}$, where $ a(b_n) $ is as in \cite[Definition 3.3.18]{embrechts2013modelling}, $ b_n = F^\leftarrow_2(1 - 1/n) $ and $ F^\leftarrow_2 $ is the quantile function. Let $ y_F \leq \infty $ be the right endpoint of the distribution $F_2$. We consider two cases, $ y_F < \infty $ and $ y_F = \infty $, separately. In the former, $ b_n \rightarrow y_F $ as $ n \rightarrow \infty $ and by \cite[Remark 2, Section 3.3]{embrechts2013modelling} $ a(b_n) \rightarrow 0 $ as $ n \rightarrow \infty $. Thus, for a large enough $ n $, we have $ g_{n2} = b_n \rightarrow y_F < \infty $ and
	\[
	\frac{g_{n2}}{g_{n0}} = \frac{y_F + o(1)}{n^{\gamma_0} L_0(n)} = o(1).
	\]
	For $ y_F = \infty $, we have $ b_n \rightarrow \infty $ and, by \cite[Remark 1, Section 3.3]{embrechts2013modelling}, $ a(b_n) = o(b_n) $. Thus, for a large enough $ n $, we have $ g_{n2} = b_n $ and
	\[
	\frac{g_{n2}}{g_{n0}} = \frac{F^\leftarrow_2(1 - 1/n)}{n^{\gamma_0} L_0(n)}.
	\]
	We continue by proof by contradiction, and assume that $\frac{F^\leftarrow_2(1 - 1/n)}{n^{\gamma_0} L_0(n)}$ does not converge to zero. Then there exists $\epsilon_0>0$ such that we can find an arbitrarily large $n$ such that
	$$
	F_2^\leftarrow(1 - 1/n) \geq \epsilon_0n^{\gamma_0} L_0(n).
	$$
	It follows that
	$$
	1-F_2\left(\epsilon_0n^{\gamma_0} L_0(n)\right) \geq \frac{1}{n}	
	$$
	and since $L_0$ is slowly varying, this further implies that
	$$
	1-F_2\left(cn^{\gamma_0} \right) \geq \frac{1}{n}
	$$
	for some constant $c>0$ and a large enough $n$. Since $\gamma_0>0$, this implies that $F_2$ is heavy tailed giving us the contradiction. This completes the proof for the case $k=2$ as well.
\end{proof}
The next result shows that the denominator in $ | |\hat{\textbf{z}}^k|_{(n-m,n)} / |\textbf{z}^k|_{(n-m,n)} - 1 | $ is negligible.

\begin{lemma}\label{lem:non_vanishing_ios}
	Let $(z_k), k=1,\ldots,n$ be an arbitrary sequence of non-negative random variables such that
	\begin{equation}
		\label{eq:uniform-zero-probability}
		\lim_{\delta \to 0} \inf_{k\geq 1}\mathbb{P}\left(z_k\geq \delta\right) = 1.
	\end{equation}
	Then, for any $\epsilon>0$ and any intermediate sequence $k_n$, there exists $\delta>0$ and $N$ such that
	$$
	\mathbb{P}(z_{(n-k_n,n)} < \delta) <\epsilon, \quad n\geq N.
	$$
\end{lemma}

\begin{proof}
	Let
	$$
	S_n(\delta) = \sum_{k=1}^n \textbf{1}_{z_k \geq \delta}.
	$$
	Then
	$$
	\mathbb{P}(z_{(n-k_n,n)} < \delta) = \mathbb{P}(S_n(\delta) < k_n).
	$$
	Indeed, $S_n(\delta) < k_n$ means that less than $k_n$ of the values are above or equal to $\delta$, which implies
	that $k_n$:th maximum of $z$ is strictly less than $\delta$. Vice versa, if $z_{(n-k_n,n)} < \delta$, then at most $k_n-1$ of values $z_k$ can be above or equal to $\delta$.
	Thus it suffices to prove that for any $\epsilon>0$, we can find $N$ and $\delta$ such that for $n\geq N$ we have
	\begin{equation*}
		\mathbb{P}(S_n(\delta) < k_n) <\epsilon.
	\end{equation*}
	Equivalently, we need to show
	\begin{equation}
		\label{eq:needed}
		\mathbb{P}(S_n(\delta) \geq k_n) > 1-\epsilon.
	\end{equation}
	Denote $\overline{S}_n(\delta) = \frac{S_n(\delta)}{n}$. By \eqref{eq:uniform-zero-probability}, for any $\widetilde{\epsilon}>0$ we can find $\delta>0$ small enough such that
	\begin{equation}
		\label{eq:needed-lower-expectation}
		\mathbb{E} \overline{S}_n(\delta)  = \frac{1}{n}\sum_{k=1}^n \mathbb{P}(z_k\geq \delta) > 1- \widetilde{\epsilon}
	\end{equation}
	uniformly in $n$.
	Together with $\overline{S}_n(\delta)\leq 1$ this gives us
	$$
	\frac{\left[\mathbb{E} \overline{S}_n(\delta)\right]^2}{\mathbb{E} \overline{S}^2_n(\delta)} \geq (1-\widetilde{\epsilon})^2
	$$
	which holds for every $n$. Next we recall the Paley-Zygmund inequality which states that, for any random variable $Z\geq 0$ with finite variance and any number $\theta\in[0,1]$, we have
	\begin{equation}
		\label{eq:paley-zygmund}
		\mathbb{P}\left(Z\geq \theta \mathbb{E} Z\right) \geq (1-\theta)^2 \frac{(\mathbb{E} Z)^2}{\mathbb{E} Z^2}.
	\end{equation}
	Since $k_n$ is an intermediate sequence, we have, applying \eqref{eq:needed-lower-expectation}, that
	$$
	\frac{k_n}{n\mathbb{E} \overline{S}_n(\delta)} \leq \frac{k_n}{n(1-\widetilde{\epsilon})} \leq 1
	$$
	provided that $n$ is large enough. Hence we may apply \eqref{eq:paley-zygmund} with $Z=\overline{S}_n(\delta)$ and $\theta = \frac{k_n}{n\mathbb{E} \overline{S}_n(\delta)}$ to compute
	\begin{equation*}
		\begin{split}
			\mathbb{P}(S_n(\delta) \geq k_n)
			&= 	\mathbb{P}\left(\overline{S}_n(\delta) \geq \frac{k_n}{n \mathbb{E} \overline{S}_n(\delta)}\mathbb{E} \overline{S}_n(\delta)\right)	\\
			&\geq \left(1- \frac{k_n}{n \mathbb{E} \overline{S}_n(\delta)}\right)^2 \frac{\left[\mathbb{E} \overline{S}_n(\delta)\right]^2}{\mathbb{E} \overline{S}^2_n(\delta)}\\
			&\geq \left(1- \frac{k_n}{n (1-\widetilde{\epsilon})}\right)^2 (1-\widetilde{\epsilon})^2.
		\end{split}
	\end{equation*}
	This implies \eqref{eq:needed}, since $\widetilde{\epsilon}>0$ can be chosen arbitrarily and independently of $n$. This concludes the proof.
\end{proof}

The next two results allow us to deduce bounds for the difference between order statistics of $|\hat{\textbf{z}}^k|$ and $|\textbf{z}^k|$.
\begin{lemma}\label{lem:monotone_order}
	Let $\textbf{a} = ( a_1, \ldots , a_n)  $ and $\textbf{b} = ( b_1, \ldots , b_n) $ satisfy $ a_i \leq b_i $, for all $ i = 1, \ldots , n $. Then
	\[
	(\textbf{a})_{(k, n)} \leq (\textbf{b})_{(k, n)},
	\]
	for all $ k = 1, \ldots , n $.
\end{lemma}
\begin{proof}
	Recall Weyl's inequality: if $ \textbf{R}, \textbf{S} \in \mathbb{R}^{n \times n}$ are symmetric matrices and $ \lambda_j(\textbf{R}) $ denotes the $ j $th largest eigenvalue of the matrix $ \textbf{R} $, $ j = 1 , \ldots , n $, then
	\[
	\lambda_{j + m}(\textbf{R}) + \lambda_{k - m}(\textbf{S}) \leq \lambda_{j}(\textbf{R} + \textbf{S}) \leq \lambda_{j - \ell}(\textbf{R}) + \lambda_{1 + \ell}(\textbf{S}).
	\]
	for all $\ell = 0, \ldots , j - 1, m = 0, \ldots k - j$, see \cite{horn1990matrix}.
	
	Let $ \mathrm{diag} (\textbf{r}) \in \mathbb{R}^{n \times n}$ denote the diagonal matrix having the elements of the vector $ \textbf{r} = (r_1, \ldots , r_n)$ as its diagonal elements. Then $ (\textbf{r})_{(k, n)} = \lambda_{n - k + 1}[\mathrm{diag}(\textbf{r})] $ and the right-hand side of Weyl's inequality with $ j = n - k + 1 $ and $ \ell = 0 $ gives,
	\begin{align}\label{eq:weyl_order_statistics}
		(\textbf{r} + \textbf{s})_{(k,n)} \leq (\textbf{r})_{(k,n)} + (\textbf{s})_{(n,n)},
	\end{align}
	for any two vectors $ \textbf{r} = (r_1, \ldots , r_n) $ and $ \textbf{s} = (s_1, \ldots , s_n) $.
	
	Apply next \eqref{eq:weyl_order_statistics} to $ \textbf{r} = \textbf{b}$ and $ \textbf{s} = \textbf{a} - \textbf{b} $ to obtain the claim,
	\[
	(\textbf{a})_{(k,n)} \leq (\textbf{b})_{(k,n)} + (\textbf{a} - \textbf{b})_{(n,n)} \leq (\textbf{b})_{(k,n)} ,
	\]
	where the second inequality holds as all elements of the sequence $ \textbf{a} - \textbf{b} $ are non-positive.
\end{proof}

\begin{lemma}\label{lem:weyl}
	Let $\textbf{x} = ( x_1, \ldots , x_n) $ and $\boldsymbol{\epsilon} = ( \epsilon_1, \ldots , \epsilon_n) $ be arbitrary. Then for all $ k = 1, \ldots , n $,
	\[
	\left| |\textbf{x} + \boldsymbol{\epsilon} |_{(k, n)} - |\textbf{x}|_{(k, n)} \right| \leq | \boldsymbol{\epsilon} |_{(n, n)},
	\]
	where for a vector $ \textbf{a} = (a_1, \ldots , a_n) $ the notation $ | \textbf{a} | \in \mathbb{R}^n$ refers to the vector of the element-wise absolute values of $ \textbf{a} $.
\end{lemma}
\begin{proof}

	Equation \eqref{eq:weyl_order_statistics} with $ \textbf{r} = |\textbf{x}| $ and $ \textbf{s} = |\boldsymbol{\epsilon}| $ in conjunction with the triangle inequality, $ |x_i + \epsilon_i| \leq |x_i| + |\epsilon_i|  $, and Lemma \ref{lem:monotone_order} allow us to estimate,
	\[
	|\textbf{x} + \boldsymbol{\epsilon}|_{(k, n)} - |\textbf{x}|_{(k, n)} \leq (|\textbf{x}| + |\boldsymbol{\epsilon}|)_{(k, n)} - |\textbf{x}|_{(k, n)} \leq  | \boldsymbol{\epsilon} |_{(n, n)},
	\]
	giving the first half of the inequality. For the other half, we have by the same set of inequalities and the expansion $ \textbf{x} = \textbf{x} + \boldsymbol{\epsilon} - \boldsymbol{\epsilon} $,
	\[
	|\textbf{x}|_{(k, n)} - |\textbf{x} + \boldsymbol{\epsilon}|_{(k, n)} \leq (|\textbf{x} + \boldsymbol{\epsilon} | + | \boldsymbol{\epsilon} | )_{(k, n)} - |\textbf{x} + \boldsymbol{\epsilon}|_{(k, n)}  \leq  | \boldsymbol{\epsilon} |_{(n, n)},
	\]
	where the second inequality is obtained by applying \eqref{eq:weyl_order_statistics} to $ \textbf{r} = |\textbf{x} + \boldsymbol{\epsilon} | $ and $ \textbf{s} = | \boldsymbol{\epsilon} |$.
	
\end{proof}

Combining previous results yields the following lemma that provides a crucial estimate for the proofs of our main theorems.

\begin{lemma}\label{lem:maximum}
	Let $k=1,\ldots,p$ be fixed.	Then, under \eqref{eq:general_model} and Assumption \ref{assu:component_rates}, we have
	\[ \max_{0 \leq m \leq k_n} \left| \frac{|\hat{\textbf{z}}^{k}|_{(n - m,n)}}{|\textbf{z}^{k}|_{(n - m,n)}}  - 1 \right| =  \mathcal{O}_p \left( \frac{ 1 }{c_{n}} \max_\ell \{ g_{n\ell} \} \right).  \]
\end{lemma}

\begin{proof}
	The left-hand side of the claim equals
	\begin{align}\label{eq:relative_difference}
		\max_{0 \leq m \leq k_n} \left| \frac{|\hat{\textbf{z}}^{k}|_{(n - m,n)} - |\textbf{z}^{k}|_{(n - m,n)}}{|\textbf{z}^{k}|_{(n - m,n)}} \right|,
	\end{align}
	where by \eqref{eq:general_model} and Lemma \ref{lem:weyl} the numerator can be bounded by
	\begin{align*}
		\left| |\hat{\textbf{z}}^{k}|_{(n - m,n)} - |\textbf{z}^{k}|_{(n - m,n)} \right| \leq \max_i \left\{ \left| \sum_{j=1}^p \hat{h}_{j} z_{ij} + \hat{r} \right| \right\} \leq \sum_{j=1}^p | \hat{h}_{j} | |\textbf{z}^{j}|_{(n,n)} + |\hat{r}|,
	\end{align*}
	where $ \hat{h}_{j} = \mathcal{O}_p(c_n^{-1}) $, $ j = 1, \ldots , p $, and $ \hat{r} = \mathcal{O}_p(c_n^{-1}) $.
	Now, by Assumption \ref{assu:component_rates}, we have
	\begin{align*}
		&\sum_{j=1}^p | \hat{h}_{j} | |\textbf{z}^{j}|_{(n,n)} + |\hat{r}| \\
		= & \sum_{j=1}^p \left( a_{nj} | \hat{h}_{j} | \frac{|\textbf{z}^{j}|_{(n,n)} - b_{nj}}{a_{nj}} +  | \hat{h}_{j} | b_{nj} \right) + |\hat{r}| \\
		= & \sum_{j=1}^p \left( \frac{a_{nj}}{c_n} \mathcal{O}_p\left( 1 \right) + \frac{b_{nj}}{c_n} \mathcal{O}_p\left( 1 \right) \right) + \mathcal{O}_p\left(\frac{1}{c_n}\right) \\
		= & \sum_{j=1}^p \mathcal{O}_p \left( \frac{g_{nj}}{c_{n}} \right) + \mathcal{O}_p\left(\frac{1}{c_n}\right) \\
		= & \mathcal{O}_p \left( \frac{ 1 }{c_{n}} \max_\ell \{ g_{n\ell} \} \right),
	\end{align*}
	where we have used the result that if one deterministic sequence eventually majorizes another, $ r_n \leq s_n $, for all $ n \geq N $, then any sequence of random variables $ x_n $ with $ x_n = \mathcal{O}_p(r_n) $ has also $ x_n = \mathcal{O}_p(s_n) $.
	
	The previous bound holds uniformly in $ m $. Thus
	\begin{align*}
		\max_{0 \leq m \leq k_n} \left| \frac{|\hat{\textbf{z}}^{k}|_{(n - m,n)} - |\textbf{z}^{k}|_{(n - m,n)}}{|\textbf{z}^{k}|_{(n - m,n)}} \right| = \mathcal{O}_p \left( \frac{ 1 }{c_{n}} \max_\ell \{ g_{n\ell} \} \right) \max_{0 \leq m \leq k_n} \frac{1}{|\textbf{z}^{1}|_{(n - m,n)}},
	\end{align*}
	where $ \max_{0 \leq m \leq k_n} |\textbf{z}^{1}|^{-1}_{(n - m,n)} = |\textbf{z}^{1}|_{(n - k_n,n)}^{-1}$ is, by Lemma \ref{lem:non_vanishing_ios}, of order $ \mathcal{O}_p(1) $. This concludes the proof.
	
\end{proof}

Finally, we end this section with the following result allowing us to handle logarithm in the estimators.

\begin{lemma}\label{lem:max_taylor}
	Let $ x_n $ be an arbitrary triangular array of random variables satisfying $ \max_{0 \leq m \leq d_n} | x_m | = \mathcal{O}_p(e_n)$ for some $ d_n $ and $ e_n = o(1)$. Furthermore, let $g: (a,b) \mapsto \mathbb{R}$ with $-\infty\leq a < 0 < b\leq \infty$ be such that $g$ is continuously differentiable at the neighbourhood of $0$. Then
	\[
	\max_{0 \leq m \leq d_n} | g( x_m ) -g(0)| = \mathcal{O}_p(e_n).
	\]
\end{lemma}
\begin{proof}
	Let $\epsilon>0$ be fixed. Then there exists $C>0$ and $N$ such that
	$$
	\mathbb{P}\left(\frac{\max_{0 \leq m \leq d_n} |x_m|}{e_n} > C\right) < \frac{\epsilon}{2}
	$$
	for $n\geq N$. By assumptions, there exists $\delta>0$ such that $g$ is continuously differentiable on an open interval $(-\delta,\delta)$. Moreover, by continuity of $g'$ we also have
	$$
	\left(g'\right)^* = \sup_{-\frac{\delta}{2}\leq x\leq \frac{\delta}{2}}|g'(x)| < \infty.
	$$
	Moreover, since $e_n = o(1)$ there exists $N^*$ such that
	$
	e_nC \leq \frac{\delta}{2}
	$
	for $n\geq N^*$.
	Thus, on the set $A_n = \{\max_{0 \leq m \leq d_n} |x_m| \leq e_n C\}$ mean value theorem implies
	$$
	\max_{0 \leq m \leq d_n}|g(x_m)-g(0)| \leq \left(g'\right)^*\max_{0 \leq m \leq d_n}|x_m|.
	$$
	Let $n\geq \max(N,N^*)$ and put $\widetilde{C}=\left(g'\right)^* C$. We have
	\begin{equation*}
		\begin{split}
			&\mathbb{P}\left(\frac{\max_{0 \leq m \leq d_n} |g(x_m)-g(0)|}{e_n} > \widetilde{C}\right) \\
			&=\mathbb{P}\left(A_n,\frac{\max_{0 \leq m \leq d_n} |g(x_m)-g(0)|}{e_n} > \widetilde{C}\right) + \mathbb{P}\left(A_n^c,\frac{\max_{0 \leq m \leq d_n} |g(x_m)-g(0)|}{e_n} > \widetilde{C}\right)\\
			&\leq \mathbb{P}\left(A_n,\frac{\left(g'\right)^*\max_{0 \leq m \leq d_n} |x_m|}{e_n} > \widetilde{C}\right) + \mathbb{P}\left(A_n^c\right)\\
			&\leq \mathbb{P}\left(\frac{\max_{0 \leq m \leq d_n} |x_m|}{e_n} > C\right) + \mathbb{P}\left(\frac{\max_{0 \leq m \leq d_n} |x_m|}{e_n} > C\right)\\
			& < \epsilon
		\end{split}
	\end{equation*}
	concluding the proof.
\end{proof}

\subsection{Convergence of the Hill and Moment estimators}
\label{subsec:main-proofs}
We begin with the proof of Theorem \ref{prop:consistency_and_limiting}.
\begin{proof}[Proof of Theorem \ref{prop:consistency_and_limiting}]
	Let $ \textbf{y} = (y_1, \ldots , y_n) \geq 0 $ and $\hat{\textbf{y}} = (\hat{y}_1, \ldots , \hat{y}_n) \geq 0$ be an arbitrary pair of samples that satisfy
	\begin{align}\label{eq:general_sample}
		\max_{0 \leq m \leq k_n} \left| \frac{(\hat{\textbf{y}})_{(n - m,n)}}{(\textbf{y})_{(n - m,n)}}  - 1 \right| =  \mathcal{O}_p \left( h_n \right),
	\end{align}
	where $ h_n = o(1) $.
	
	Recall that the Hill estimator is given by
	\[
	\hat{\gamma}_H(\textbf{y}) = M_n^{(1)}(\textbf{y}) =  \frac{1}{k_n} \sum_{m=0}^{k_n - 1} \log  \frac{(\textbf{y})_{(n-m, n)}}{(\textbf{y})_{(n-k_n, n)}},
	\]
	where $k_n/n \to 0$, $k_n \to \infty$.
	In the proof, we use the short notation
	\[
	\hat{w}_m := \frac{(\hat{\textbf{y}})_{(n-m,n)}}{(\textbf{y})_{(n-m,n)}} - 1. %  \quad \mbox{and} \quad m_n = \underset{1 \leq m \leq k_n}{\mathrm{argmax}} \left| \frac{(\hat{\textbf{y}})_{(n - m,n)}}{(\textbf{y})_{(n - m,n)}}  - 1 \right| .
	\]
	We now have
	\begin{align*}
		\left| M_n^{(1)}(\hat{\textbf{y}}) - M_n^{(1)}(\textbf{y}) \right| &= \left| \frac{1}{k_n} \sum_{m=0}^{k_n - 1} \left[ \log \frac{(\hat{\textbf{y}})_{(n-m,n)}}{(\hat{\textbf{y}})_{(n-k_n,n)}} - \log \frac{(\textbf{y})_{(n-m,n)}}{(\textbf{y})_{(n-k_n,n)}} \right] \right|\\
		&= \left| \frac{1}{k_n} \sum_{m=0}^{k_n - 1} \left[ \log (1 + \hat{w}_m) - \log (1 + \hat{w}_{k_n}) \right] \right| \\
		&\leq  \frac{1}{k_n} \sum_{m=0}^{k_n - 1} \left| \log (1 + \hat{w}_m) \right| + \left| \log (1 + \hat{w}_{k_n}) \right| \\
		&\leq 2	\max_{0 \leq m \leq k_n} \left|  \log (1 + \hat{w}_m) \right|.
	\end{align*}
	The assumptions of Lemma \ref{lem:max_taylor} are now satisfied for $ x_n = \hat{w}_n $, $ d_n = k_n $, $ e_n = h_n $ and $ g(x) = \log(1 + x) $, implying that $ | M_n^{(1)}(\hat{\textbf{y}}) - M_n^{(1)}(\textbf{y}) | = \mathcal{O}_p(h_n) $. Plugging in $\textbf{y} = \textbf{z}^k$ and $\hat{\textbf{y}} = \hat{\textbf{z}}^k$, and using Lemma \ref{lem:maximum}, now give the convergence of the Hill estimator. For the moment estimator, recall that
	$$
	\hat{\gamma}_M(\textbf{y}) = M_n^{(1)}(\textbf{y}) + 1 - \frac{1}{2}\left( 1-\frac{[M_n^{(1)}(\textbf{y})]^2}{M_n^{(2)}(\textbf{y})} \right)^{-1}.
	$$
	By the first part of the proof, we have
	\begin{equation}
		\label{eq:hill-rate}
		|M_n^{(1)}(\hat{\textbf{y}})-M_n^{(1)}(\textbf{y})| = \mathcal{O}_p(h_n).
	\end{equation}
	It thus suffices to prove that
	\begin{equation}
		\label{eq:moment-rate}
		\left|\frac{[M_n^{(1)}(\textbf{y})]^2}{M_n^{(2)}(\textbf{y})} -\frac{[M_n^{(1)}(\hat{\textbf{y}})]^2}{M_n^{(2)}(\hat{\textbf{y}})}\right| = \mathcal{O}_p \left(\frac{h_n}{\hat{\gamma}_H(\textbf{y})}\right).
	\end{equation}
	Indeed, since $M_n^{(1)}(\textbf{y}) = \hat{\gamma}_H(\textbf{y})$ as a convergent sequence is uniformly tight, i.e., $\mathcal{O}_p(1)$, it follows from the convergence of $\hat{\gamma}_M(\textbf{y})$ that
	$$
	\left( 1-\frac{[M_n^{(1)}(\textbf{y})]^2}{M_n^{(2)}(\textbf{y})} \right)^{-1} = \mathcal{O}_p(1).
	$$
	Then \eqref{eq:moment-rate} together with the assumption $\frac{h_n}{\hat{\gamma}_H(\textbf{y})} \rightarrow_p 0$ implies that also
	$$
	\left( 1-\frac{[M_n^{(1)}(\hat{\textbf{y}})]^2}{M_n^{(2)}(\hat{\textbf{y}})} \right)^{-1} = \mathcal{O}_p(1).
	$$
	The claim then follows by using
	$$
	(1-a)^{-1} - (1-b)^{-1} = \frac{a-b}{1-a}(1-b)^{-1}, \quad a,b\in(0,1)
	$$
	with $a = \frac{[M_n^{(1)}(\textbf{y})]^2}{M_n^{(2)}(\textbf{y})}$ and $b= \frac{[M_n^{(1)}(\hat{\textbf{y}})]^2}{M_n^{(2)}(\hat{\textbf{y}})}$, leading to
	\begin{equation}
		\label{eq:moment-est-rate}
		\left|\hat{\gamma}_M(\hat{\textbf{y}}) -\hat{\gamma}_M(\textbf{y})\right| = \mathcal{O}_p \left(\frac{h_n}{\hat{\gamma}_H(\textbf{y})}\right).
	\end{equation}
	In order to prove \eqref{eq:moment-rate} we write
	\begin{align*}
		\left|\frac{[M_n^{(1)}(\textbf{y})]^2}{M_n^{(2)}(\textbf{y})} -\frac{[M_n^{(1)}(\hat{\textbf{y}})]^2}{M_n^{(2)}(\hat{\textbf{y}})}\right| &\leq \frac{1}{M_n^{(2)}(\textbf{y})}\left|[M_n^{(1)}(\textbf{y})]^2-M_n^{(1)}(\hat{\textbf{y}})]^2\right| \\
		&+ \frac{[M_n^{(1)}(\hat{\textbf{y}})]^2}{M_n^{(2)}(\hat{\textbf{y}})M_n^{(2)}(\textbf{y})}\left|M_n^{(2)}(\textbf{y}) -M_n^{(2)}(\hat{\textbf{y}})\right| \\
		&=:I_1(n) + I_2(n).
	\end{align*}
	For the first term $I_1(n)$, we use $a^2-b^2 = (a-b)(a+b)$ and \eqref{eq:hill-rate} to get
	\begin{align*}
		\left|[M_n^{(1)}(\textbf{y})]^2-M_n^{(1)}(\hat{\textbf{y}})]^2\right| & = \left|M_n^{(1)}(\textbf{y})-M_n^{(1)}(\hat{\textbf{y}})\right|\left|M_n^{(1)}(\textbf{y})+M_n^{(1)}(\hat{\textbf{y}})\right|\\
		&\leq \left|M_n^{(1)}(\textbf{y})-M_n^{(1)}(\hat{\textbf{y}})\right|^2 + 2\left|M_n^{(1)}(\textbf{y})-M_n^{(1)}(\hat{\textbf{y}})\right| M_n^{(1)}(\textbf{y})\\
		&=\mathcal{O}_p\left(h_nM_n^{(1)}(\textbf{y})\right).
	\end{align*}
	Here we used also the fact that $\frac{h_n}{M_n^{(1)}(\textbf{y})} \rightarrow_p 0$. Moreover, by Cauchy-Schwarz inequality we have $[M_n^{(1)}(\textbf{y})]^2 \leq M_n^{(2)}(\textbf{y})$. Thus we can estimate
	\begin{align*}
		I_1(n) &= \frac{1}{M_n^{(2)}(\textbf{y})}\left|[M_n^{(1)}(\textbf{y})]^2-M_n^{(1)}(\hat{\textbf{y}})]^2\right| \\
		& \leq \frac{1}{[M_n^{(1)}(\textbf{y})]^2} \mathcal{O}_p\left(h_nM_n^{(1)}(\textbf{y})\right) \\
		&= \mathcal{O}_p\left(\frac{h_n}{M_n^{(1)}(\textbf{y})}\right)
	\end{align*}
	which, by recalling that $\hat{\gamma}_H(\textbf{y})=M_n^{(1)}(\textbf{y})$, gives the claim for the term $I_1(n)$. For the term $I_2(n)$, we apply
	$a^2-b^2 = (a-b)(a+b)$ again yielding
	\begin{align*}
		&\left| M_n^{(2)}(\hat{\textbf{y}}) - M_n^{(2)}(\textbf{y}) \right| \\
		&= \left| \frac{1}{k_n} \sum_{m=0}^{k_n - 1} \left[\left[ \log \frac{(\hat{\textbf{y}})_{(n-m,n)}}{(\hat{\textbf{y}})_{(n-k_n,n)}}\right]^2 - \left[\log \frac{(\textbf{y})_{(n-m,n)}}{(\textbf{y})_{(n-k_n,n)}} \right]^2\right] \right|\\
		&\leq \frac{1}{k_n} \sum_{m=0}^{k_n - 1} \left| \log(1 + \hat{w}_m) - \log (1 + \hat{w}_{k_n}) \right|\left| \log \frac{(\hat{\textbf{y}})_{(n-m,n)}}{(\hat{\textbf{y}})_{(n-k_n,n)}} +\log \frac{(\textbf{y})_{(n-m,n)}}{(\textbf{y})_{(n-k_n,n)}} \right| \\
		&\leq \frac{2	\max_{0 \leq m \leq k_n} \left|  \log (1 + \hat{w}_m) \right|}{k_n} \sum_{m=0}^{k_n - 1} \left| \log \frac{(\hat{\textbf{y}})_{(n-m,n)}}{(\hat{\textbf{y}})_{(n-k_n,n)}} +\log \frac{(\textbf{y})_{(n-m,n)}}{(\textbf{y})_{(n-k_n,n)}} \right| .
	\end{align*}
	Here
	\begin{align*}
		&\frac{1}{k_n} \sum_{m=0}^{k_n - 1} \left| \log \frac{(\hat{\textbf{y}})_{(n-m,n)}}{(\hat{\textbf{y}})_{(n-k_n,n)}} +\log \frac{(\textbf{y})_{(n-m,n)}}{(\textbf{y})_{(n-k_n,n)}} \right| \\
		& \leq \frac{1}{k_n} \sum_{m=0}^{k_n - 1} \left| \log \frac{(\hat{\textbf{y}})_{(n-m,n)}}{(\hat{\textbf{y}})_{(n-k_n,n)}} -\log \frac{(\textbf{y})_{(n-m,n)}}{(\textbf{y})_{(n-k_n,n)}} \right| \\
		&	+\frac{2}{k_n} \sum_{m=0}^{k_n - 1} \log \frac{(\textbf{y})_{(n-m,n)}}{(\textbf{y})_{(n-k_n,n)}} \\
		&\leq 2	\max_{0 \leq m \leq k_n} \left|  \log (1 + \hat{w}_m) \right| + 2M_n^{(1)}(\textbf{y}).	\end{align*}
	Together with Lemma \ref{lem:max_taylor} this gives us
	$$
	\left| M_n^{(2)}(\hat{\textbf{y}}) - M_n^{(2)}(\textbf{y}) \right|	\leq \mathcal{O}_p\left(h_nM_n^{(1)}(\textbf{y})\right).
	$$
	Applying Cauchy-Schwarz again to get $[M_n^{(1)}(\hat{\textbf{y}})]^2 \leq M_n^{(2)}(\hat{\textbf{y}})$
	gives us
	\begin{align*}
		I_2(n) &= \frac{[M_n^{(1)}(\hat{\textbf{y}})]^2}{M_n^{(2)}(\hat{\textbf{y}})M_n^{(2)}(\textbf{y})}\left|M_n^{(2)}(\textbf{y}) -M_n^{(2)}(\hat{\textbf{y}})\right| \\
		& \leq \frac{1}{[M_n^{(1)}(\textbf{y})]^2}\left|M_n^{(2)}(\textbf{y}) -M_n^{(2)}(\hat{\textbf{y}})\right|  \\
		&= \mathcal{O}_p\left(\frac{h_n}{M_n^{(1)}(\textbf{y})}\right).
	\end{align*}
	Plugging in $\textbf{y} = \textbf{z}^k$ and $\hat{\textbf{y}} = \hat{\textbf{z}}^k$, and using Lemma \ref{lem:maximum}, now give the convergence of the moment estimator. This completes the proof.
\end{proof}
Applying the above computations, the proof of Theorem \ref{prop:consistency_and_limiting2} is now rather simple.
\begin{proof}[Proof of Theorem  \ref{prop:consistency_and_limiting2}]
	We write
	\[
	\sqrt{k_n} \left( \hat{\gamma}_H(|\hat{\textbf{z}}^1|) - C_H \right) = \sqrt{k_n} \left( \hat{\gamma}_H(|\hat{\textbf{z}}^1|) - \hat{\gamma}_H(|{\textbf{z}}^1|) \right) + \sqrt{k_n} \left( \hat{\gamma}_H(|{\textbf{z}}^1|) - C_H \right).
	\]
	The first claim now follows directly from \eqref{eq:hill-rate}. Similarly, the second claim follows directly from
	\[
	\sqrt{k_n} \left( \hat{\gamma}_M(|\hat{\textbf{z}}^1|) - C_M \right) = \sqrt{k_n} \left( \hat{\gamma}_M(|\hat{\textbf{z}}^1|) - \hat{\gamma}_M(|{\textbf{z}}^1|) \right) + \sqrt{k_n} \left( \hat{\gamma}_M(|{\textbf{z}}^1|) - C_M \right)
	\]
	together with \eqref{eq:moment-est-rate}.
\end{proof}

\section{Auxiliary simulation}\label{sec:appendix_simulation}

\renewcommand\thefigure{\thesection\arabic{figure}}

\textcolor{black}{The auxiliary simulation is otherwise similar to the simulation study in the main text, but is conducted with i.i.d. vectors instead of time-dependent series, putting us in the context of Section \ref{subsec:ICA} }

\textcolor{black}{Let $ \tilde{\textbf{z}}_1, \ldots , \tilde{\textbf{z}}_n $ be a collection of i.i.d. random vectors whose marginal distributions are independent and distributed as
	\begin{align*}
		\tilde{\textbf{z}}_j = \begin{pmatrix} \texttt{Pareto}(5) & \texttt{Pareto}(15) &  \texttt{Pareto}(30)   \end{pmatrix}^\top,
	\end{align*}
	where $\texttt{Pareto}(\alpha)$ denotes the Pareto distribution with shape parameter $\alpha$, scale parameter~1 and location parameter 0. Hereby, in this simulation setting, the first component has the heaviest tail and the corresponding theoretical extreme value index is 1/5.}

\textcolor{black}{The simulation was conducted with six distinct sample sizes $n$, which were $300,$ $10^3,$ $10^4,$ $10^5,$ $10^6$ and $10^7$. The threshold sequence $k_n$ was again chosen to be $k_n = \lfloor n^{1/4} \rfloor$ and, for each sample size, the simulation was iterated 2000 times.}

\textcolor{black}{As a preliminary step, the simulated observations $\tilde{\textbf{z}}_j $ were centered. Here, the centered observations are denoted as ${\textbf{z}}_i$. In every iteration $h \in \{1,\ldots 2000\}$, we applied the following linear transformation,
	\begin{align*}
		\textbf{x}_i = \boldsymbol{\Omega}_h {\textbf{z}}_i, \qquad \forall i\in \{1,\ldots,n\},
	\end{align*}
	where the elements of the $\mathbb{R}^{3\times 3}$-matrix $\boldsymbol{\Omega}$ were simulated independently, and separately in every iteration, from the univariate uniform distribution $\texttt{unif}(-100,100)$.}

We then applied the FastICA procedure, implemented in the R package fICA \citep{Rfica}, to the mixed observations $\textbf{x}_1, \ldots , \textbf{x}_n$. Note that the asymptotic convergence of FastICA requires that all of the components have finite fourth moments. In this simulation study, the existence of the required moments is satisfied, as even the most heavy tailed component $\texttt{Pareto}(5)$ has finite fourth moments.

For a small number of iterations, the FastICA algorithm failed to converge. In the case of a failed FastICA convergence, the observations and the mixing matrix were simulated again, until 2000 successful FastICA estimates were obtained. We denote the observations unmixed by FastICA as $ \hat{\textbf{z}}_1, \ldots , \hat{\textbf{z}}_n $. In the sequel, we use the notations $|\hat{\textbf{z}}|$ and $|{\textbf{z}}|$ for the sets $\{ | \hat{\textbf{z}}_1 |, \ldots , | \hat{\textbf{z}}_n | \}$ and $\{ | \textbf{z}_1 |, \ldots , | \textbf{z}_n | \}$, respectively. Here, the absolute value of a vector is taken elementwise.

In this simulation study, we have $\max_\ell ({g_{n\ell}}) = n^{-1/5}$, which corresponds to the \texttt{Pareto}(5) distribution. Furthermore, the FastICA unmixing estimator is $\sqrt{n}$-consistent, which gives $c_n = \sqrt{n}$. Hereby, under our choice of $k_n = \lfloor n^{1/4} \rfloor$, the assumptions required by \textcolor{black}{Theorems \ref{prop:consistency_and_limiting} and  \ref{prop:consistency_and_limiting2}} hold. Thus, under large sample sizes, the independent component estimation should have a negligible effect on the \textcolor{black}{extreme value} index estimation. \textcolor{black}{This implies that}, for large sample sizes, the \textcolor{black}{extreme value} index estimates calculated from $|\hat{\textbf{z}}|$ and $|{\textbf{z}}|$ are expected to be close to each other.

We estimated the \textcolor{black}{extreme value} indices for every component from \textcolor{black}{both} $|\hat{\textbf{z}}|$ and $|{\textbf{z}}|$, using both the Hill estimator and the moment estimator. Note that, both the Hill and the moment estimator produce three \textcolor{black}{extreme value} index estimates, one for each component. \textcolor{black}{Thus, in every simulation iteration, we again collected} the largest of the three \textcolor{black}{extreme value} index estimates, denoted in the following by $ \hat{\gamma}(|\hat{\textbf{z}}|) $ and $ \hat{\gamma}(|{\textbf{z}}|) $. Their histograms for sample sizes $n=300,10^3,10^4$ are displayed in Figure \ref{fig:simu_1}.

\begin{figure}
	\includegraphics[width=1\textwidth]{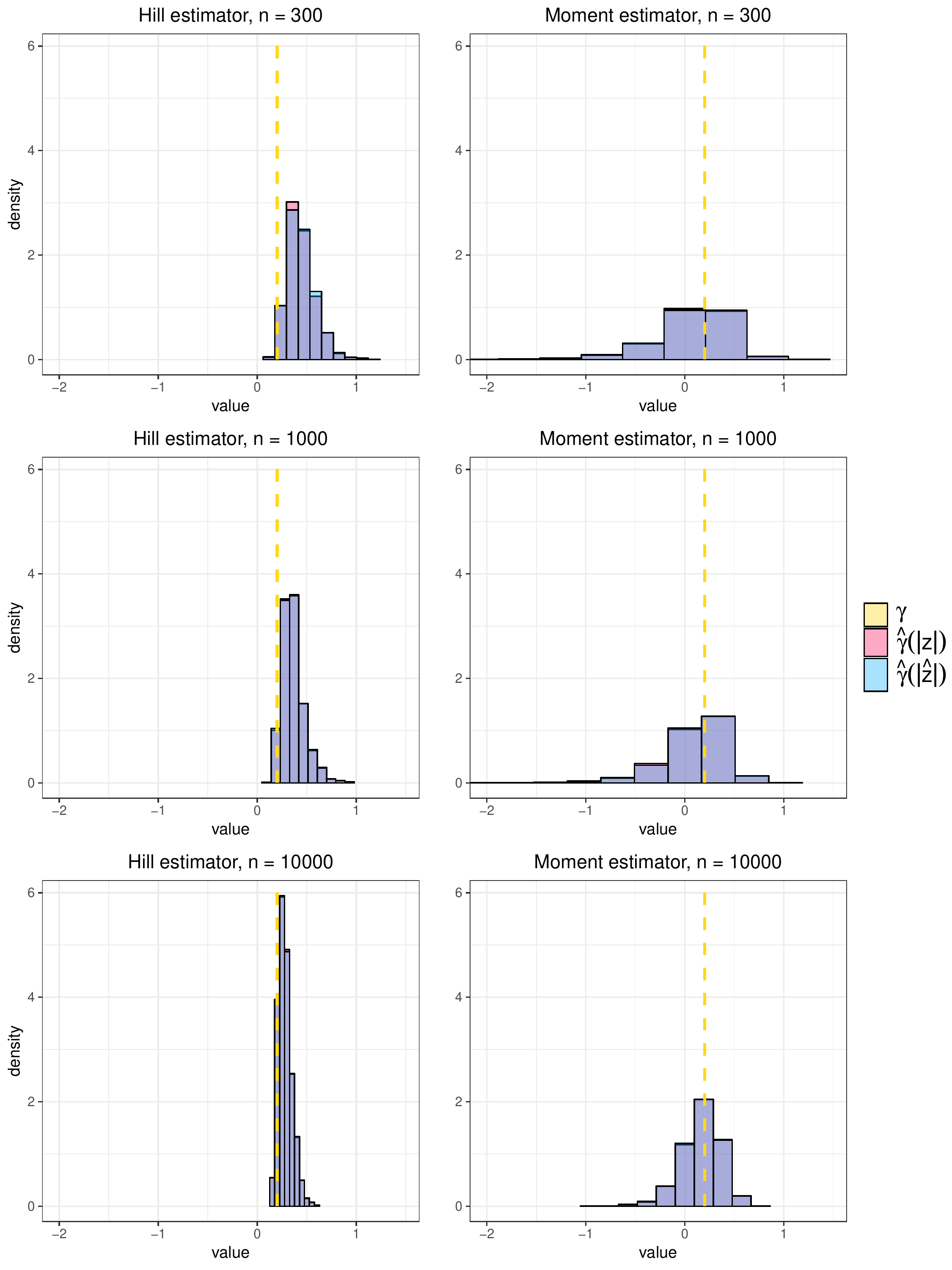}
	\caption{Histograms of $ \hat{\gamma}(|{\textbf{z}}|) $ (light red) and $ \hat{\gamma}(|\hat{\textbf{z}}|) $ (light blue) in \textcolor{black}{the auxiliary simulation study} with sample sizes 300, 1000 and 10 000. The dashed yellow vertical line is the theoretical extreme value index $\gamma = 1/5$. The dark blue color in the histograms represents the area, where the two histograms overlap.  }
	\label{fig:simu_1}
\end{figure}

In Figure \ref{fig:simu_1}, the \textcolor{black}{extreme value} indices estimated from $|\hat{\textbf{z}}|$ are illustrated using light blue colour, and the \textcolor{black}{extreme value} index estimates calculated from the original $|{\textbf{z}}|$ are illustrated using light red colour. Furthermore, the dark blue colour illustrates the proportion of estimates that overlap and the dashed yellow vertical line represents the theoretical \textcolor{black}{extreme value} index value $\gamma = 1/5$. Values smaller than $-2$ are omitted from the figure, as only a total of 5 moment estimates were smaller than $-2$.

In Figure \ref{fig:simu_1}, already with sample size $n=300$, the two histograms overlap almost completely. Moreover, when sample size is $n=1000$ or larger, one cannot visually distinguish the two histograms from each other. This illustrates that for sample sizes  $n=1000$ or larger, the effect of the ICA step is close to negligible. Hereby, we have omitted the histograms corresponding to sample sizes $10^5,10^6$ and $10^7$, as they carry no new information.

When comparing the \textcolor{black}{performances} of the Hill estimator and the moment estimator, Figure \ref{fig:simu_1} indicates that the variance of the moment estimator is larger of the two. However, the moment estimates seem to be more evenly centered around the true $\gamma$. On the other hand, the Hill estimator \textcolor{black}{seems} to be slightly biased, as is expected \cite{de2007extreme}. The bias seems to decrease as the sample size increases. \textcolor{black}{The results are thus overall largely similar to the time series example in the main text.}

\begin{figure}
	\includegraphics[width=1\textwidth]{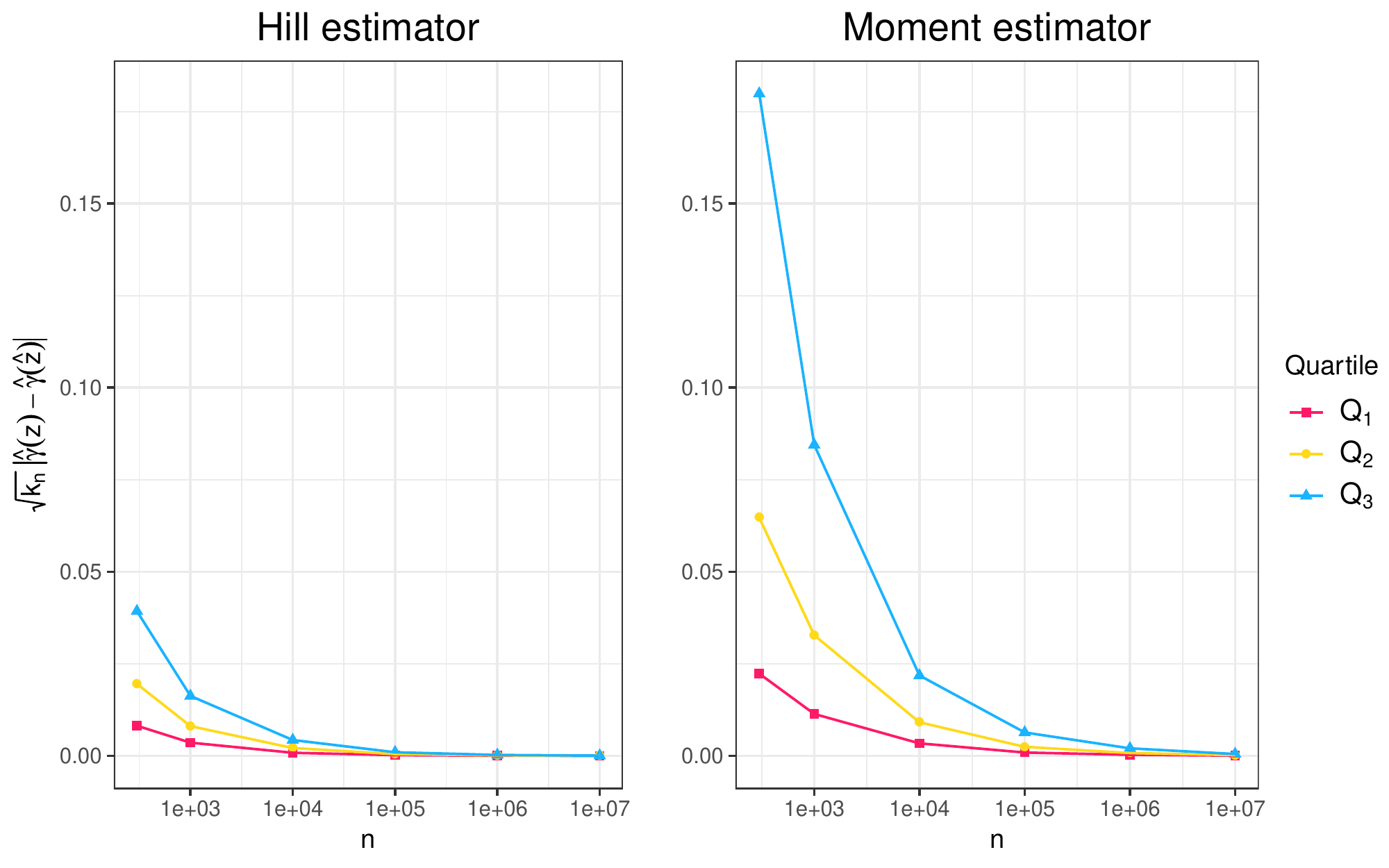}
	\caption{The quartiles of $\sqrt{k_n}|\hat{\gamma}(|\textbf{z}|) -\hat{\gamma}(|\hat{\textbf{z}}|) |$ for the Hill estimator and the  moment estimator in \textcolor{black}{the auxiliary simulation study.}}
	\label{fig:simu_2}
\end{figure}

Figure \ref{fig:simu_2} illustrates the absolute differences, scaled with $\sqrt{k_n}$, between the \textcolor{black}{extreme value} index estimates calculated from $|\textbf{z}|$ and $|\hat{\textbf{z}}|$. In Figure \ref{fig:simu_2}, the red and blue curves represent the first and third empirical quartiles, respectively. Additionally, the yellow curve is the corresponding sample median curve. The differences seem to converge to zero for both the Hill and the moment estimator. However, the moment estimator seems to require larger sample sizes \textcolor{black}{for the convergence}. The quartile $Q_3$ that corresponds to the Hill estimator is close to zero with sample sizes larger or equal to $10^5$. Conversely, in this simulation study, the moment estimator quartile $Q_3$ requires samples of size $10^7$ in order for it to be equally close to zero.% when compared to Hill estimator $Q_3$ with sample sizes~$10^5$.

\section{Additional figures for the real data example}\label{sec:appendixC}

\textcolor{black}{Figure \ref{fig:example_1} shows the original four-variate time series $ \textbf{x}_i $ analysed in the real data example in Section \ref{sec:real_data} of the main text. Figure \ref{fig:example_3} shows the four latent series estimated from the four series in Figure \ref{fig:example_1} with generalized SOBI \citep{miettinen2019extracting}.}

\begin{figure}[htp]
	\includegraphics[width=1\textwidth]{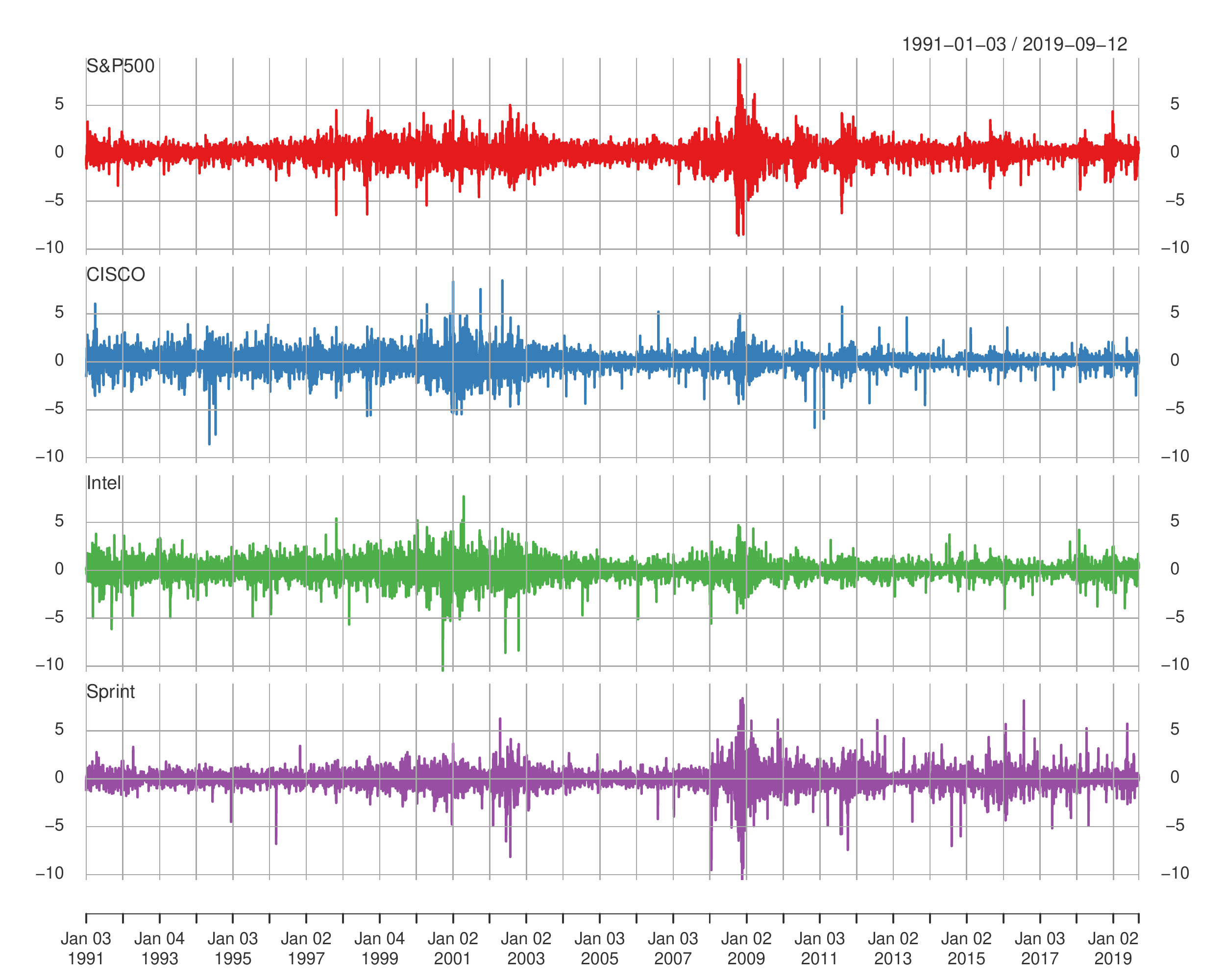}
	\caption{The log-returns of the S\&P500 index and the stock prices of CISCO Systems, Intel Corporation and Sprint Corporation in the period of January 3rd, 1991 -- September 12th, 2019.}
	\label{fig:example_1}
\end{figure}

\begin{figure}[htp]
	\includegraphics[width=1\textwidth]{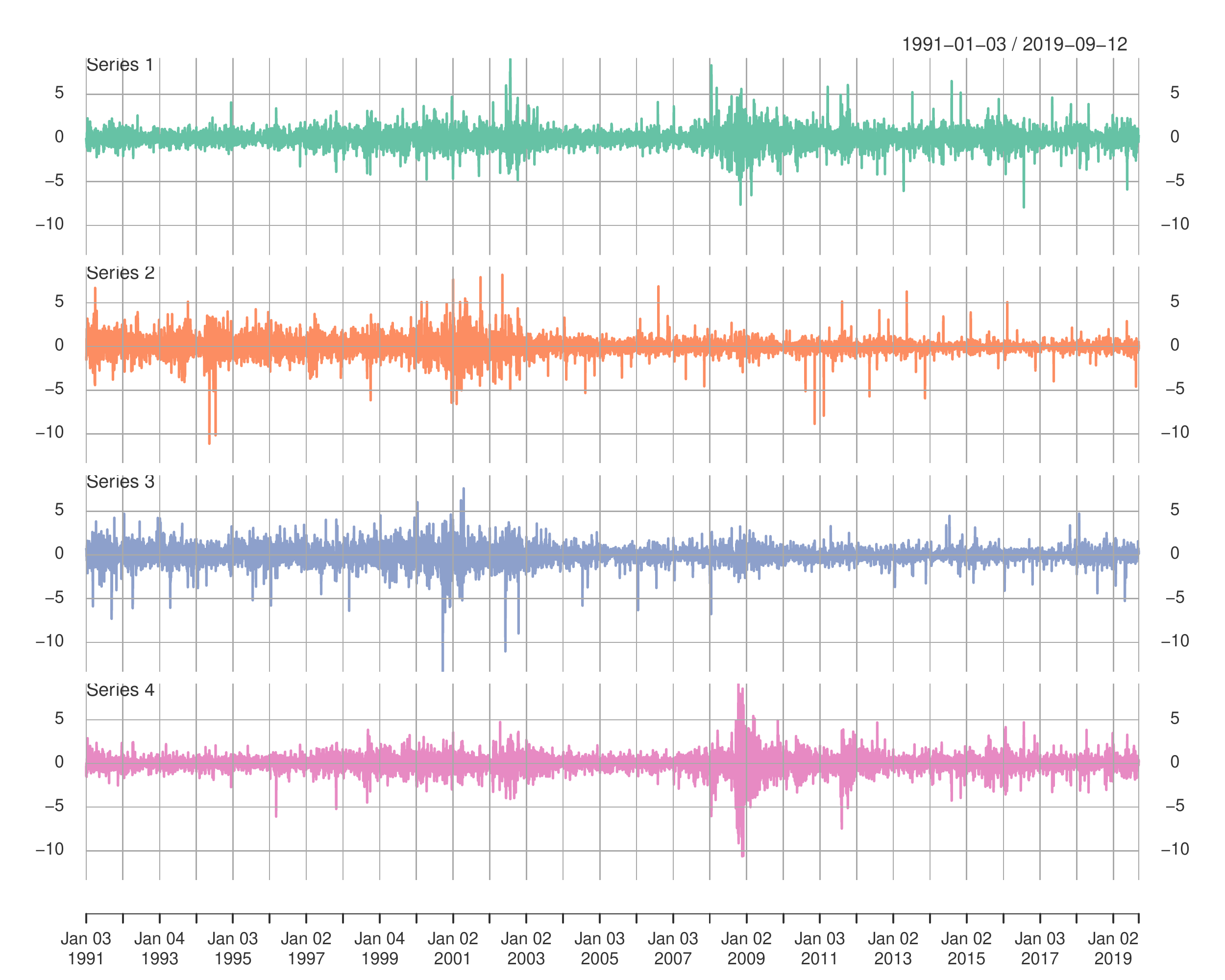}
	\caption{The latent source series estimated from the log-return observations with generalized SOBI.}
	\label{fig:example_3}
\end{figure}

\bibliographystyle{chicago}

\bibliography{references}
\end{document}